\documentclass[12pt,reqno, a4paper]{amsart}
\usepackage{amsmath}
\usepackage{mathrsfs}
\usepackage{amssymb}
  \usepackage{paralist}
  \usepackage{graphics} 
  \usepackage{epsfig} 
\usepackage{graphicx}
\usepackage{pgfplots}
\pgfplotsset{compat=1.5}
\usepackage{epstopdf}
\usepackage[colorlinks=true]{hyperref}
\usepackage{dsfont}
\usepackage{float}

\newtheorem{theorem}{Theorem}[section]

\newtheorem{lemma}[theorem]{Lemma}
\newtheorem{proposition}[theorem]{Proposition}

\theoremstyle{definition}

\newtheorem{remark}[theorem]{Remark}

\newcommand{\ep}{\varepsilon}
\newcommand{\eps}[1]{{#1}_{\varepsilon}}

\let\bb\mathbb
\def \R {\mathbb{R}}
\def \N {\mathbb{N}}
\def \T {\mathbb{T}}
\def \L {\mathbb{L}}
\def \Z{\mathbb{Z}}

\newcommand\RR{{{\mathbb R}}}
\newcommand\NN{{{\mathbb N}}}
\newcommand\cF{{\mathcal F}}

\numberwithin{equation}{section}

\begin{document}
\title[Numerical computation for Boltzmann equation] 
      {Numerical computation for the non-cutoff radially symmetric homogeneous Boltzmann equation}

\date{\today}
\author[L. Glangetas, I. Jrad ]{}

%
%
\email{leo.glangetas@univ-rouen.fr} 
\email{brahimjrad92@hotmail.com}

\thanks{The second author is supported by a grant from Lebanon}


\maketitle

\centerline{\scshape L\'eo Glangetas, Ibrahim Jrad}
\medskip
\medskip
{\footnotesize
 \centerline{\it{Universit\'e de Rouen Normandie, UMR 6085-CNRS, Math\'ematiques}}
   \centerline{\it{Avenue de l'Universit\'e, BP.12, 76801 Saint Etienne du Rouvray, France}}
} 


\begin{abstract}
For the non cutoff radially symmetric homogeneous Boltzmann equation
with Maxwellian molecules, 
we give the numerical solutions using symbolic manipulations and 
spectral decomposition of Hermit functions.
The initial data can belong to some measure space. 
\end{abstract}

\tableofcontents


\section{Introduction}

\subsection{The Boltzmann equation}

The Boltzmann equation, derived by Boltzmann in 1872 (and Maxwell 1866), 
models the behavior of a dilute gas (see \cite{Boltz}).
As we know, Boltzmann has created a theory which described the movement 
of gases as balls which could bump and rebound against each other \cite{C, DP14}. 
This model can be considered by one of many cases which represent the so-called kinetic equation. 
Presently, the diversity of sciences and applications contains these models 
such as rarefied gas dynamics, semiconductor modeling, radiative transfer, 
and biological and social sciences. 
This type of equations is made by including a combination of a linear transport term and several interaction terms which provide the time evolution of the distribution 
of particles in the phase space. 
The equation that bears his name is the following
\begin{equation*} 
\partial_t f + v.\nabla_{x} f ={\bf {Q}}(f, f)
\end{equation*}
where  
$f = f(t,x,v) \geq 0$ is the probability density to find a particle at the time $t$, 
on the position $x$ and with velocity $v$
where the physical and the velocity space are located in three dimensions. 
The term $v.\nabla_{x} f$ describes the free action of particles and
${\bf {Q}}(f, f)$ is a bilinear operator which describes the binary collision process.
It is called the Boltzmann collision operator and given by
\begin{equation*}
{\bf {Q}}(g,f)(v)
=\int_{\mathbb{R}^{3}}
\int_{\mathbb{S}^{2}}
B(v-v_{\ast},\sigma)
(g(v_{\ast}^{\prime})f(v^{\prime})-g(v_{\ast})f(v))
dv_{\ast}d\sigma
\end{equation*}
where for $\sigma\in \mathbb{S}^{2}$, the symbols $v_{\ast}^{\prime}$ 
and $v^{\prime}$ are abbreviations for the expressions,
$$
v^{\prime}=\frac{v+v_{\ast}}{2}+\frac{|v-v_{\ast}|}{2}\sigma,\,\,\,\,\, v_{\ast}^{\prime}
=\frac{v+v_{\ast}}{2}-\frac{|v-v_{\ast}|}{2}\sigma,
$$
which are obtained in such a way that collision preserves momentum and kinetic energy, namely
$$
v_{\ast}^{\prime}+v^{\prime}=v+v_{\ast},\,\,\,\,\, 
|v_{\ast}^{\prime}|^{2}+|v^{\prime}|^{2}=|v|^{2}+|v_{\ast}|^{2} 
$$
where $|\cdot|$ is the Euclidean norm on $\mathbb{R}^3$.
Note that $v$, $v^{\prime}$ are the velocities before collision 
and $v_{\ast}$, $v_{\ast}^{\prime}$ the velocities after collision.  

The non-negative cross section $B(z,\sigma)$ depends only on $|z|$ 
and the scalar product $\frac{z}{|z|}\cdot\sigma = \cos\theta$ 
where $\theta$ is the deviation angle. Without loss of generality, 
we may assume that this cross section is supported on the set $\cos\theta\geq0$.
See for instance \cite{NYKC1} for more details on the cross section 
and \cite{Villani} for a general collision kernel.
For physical models, it usually takes the form
$$
B(v-v_{\ast},\sigma)=\Phi(|v-v_{\ast}|)b(\cos\theta),~~~~
\cos\theta=\frac{v-v_{\ast}}{|v-v_{\ast}|}\cdot\sigma,~~ 0\leq\theta\leq\frac{\pi}{2},
$$
where $\Phi(|v-v_{\ast}|) = |v-v_{\ast}|^{\gamma}$ is a kinetic factor 
and $\gamma > -3$.

In this work, we consider the spatially homogeneous case, that means the density distribution   
 $f=f(t,v)$ depends on the variables $t\geq0$, $v\in\mathbb{R}^{3}$
 and is uniform with respect to $x$.
 So that the Boltzmann equation reads as
\begin{equation}\label{eq Boltz f}
\left\{
\begin{array}{ll}
   \partial_t f= {\bf {Q}}(f,f),\\
  f(0,v)=F(v)
\end{array}
\right.
\end{equation}
where the initial data $F$ is depends only on $v$.
For the collision kernel, we study only the Maxwellian molecules and 
non-cutoff cases (see \cite{ Desv96, Desv97, DG99, DG2000, DW, GS, NYKC14}), that means
the kinetic factor $\Phi\equiv 1$ and 
\begin{equation}\label{non-cutoff-intro}
  b(\cos\theta) \approx \frac{1}{|\theta|^{2+2s}}, 
  \quad 0<s<1, \,\, \theta\in  \bigl(0,\frac{\pi}{2}\bigr].
\end{equation}
\subsection{Results on the Boltzmann equation}
With the previous assumption (the non-cutoff case) on the cross-section, 
there is existence of a weak solution for the Boltzmann equation \eqref{eq Boltz f}
for a positive initial value $F \in L^{1}_{2+\delta}(\RR^3)$. 
See \cite{V98} and many others.

Moreover, it is well-known that there is a regularization effect in Sobolev 
and Schwartz or analytic spaces for any time $t>0$ 
(we refer the reader to \cite{Desv96, Desv97} and recently \cite{AMUXY})
and that the solutions converge to the Gaussian when the time tends to infinity (\cite{GS}). 

An important point that our distribution lives in a multidimensional space: 
this reason make us think that we have a numerical problem because in this case 
the computational cost is more or less forbidden \cite{DP14}. 
The study of the numerical part for kinetic equations is not obvious 
due to many difficulties come from the computational cost. 
To clarify more, we mention two of these difficulties:
It is clear the appearing of multiple scales, and then to get out of the resolution 
of the stiff dynamics, one should build suitable numerical schemes 
\cite{Ji95, Ji99, BLM, PD, Ji12, DP13}. 
The other one is that the collision operator is defined by multidimensional integrals 
and to compute one should solve it point by point as physical space \cite{PR, FMP}.
To treat kinetic equations numerically, there is several ways which are used over 
the centuries until now: 
probabilistic numerical methods such as 
Direct Simulation Monte Carlo (DSMC) schemes \cite{C, Bi}, and, 
deterministic numerical methods such as finite volume, semi-Lagrangian 
and spectral schemes \cite{DP14}. 

There are two important deterministic methods which are used in the past decades : 
the discrete velocity method (DVM) \cite{GSJ, FJ, AAJ, BCD98, VA, ZYY} and 
the Fourier spectral method (FSM) \cite{AS, LB, PR, GHHH16, ZYY}. Due to its discrete nature, 
the DVM preserves positivity of the distribution function, the H-theorem 
and the exact conservation of mass, energy and momentum.
Note that the Fourier spectral method is based on two main things : 
the truncation of the collision operator and 
the restriction of the distribution function to an appropriate cube,  
for more details see \cite{PR, MP}. 
   
Our goal is to present an alternative method to solve formally and numerically 
the homogeneous Boltzmann equation in the non-cutoff case.
In this work, we consider the radial symmetric case
and we use a spectral method : 
we first compute the spectral coefficients of the solution 
with a formal computation software 
(Maple$^{\mbox{\scriptsize{\textregistered}}}$13; the codes can be provided). 
We then approximate these exact solutions and check the numerical results.  

The used method helps us to motivate our work in several ways: 
It let us in the physical view understand more the behavior of the solutions and
as we compute the first exact projections of the solutions on the spectral basis, 
that is in the numerical view, some other algorithms can be tested in the non-cutoff case 
(recall that the explicit 2D ``BKW'' solutions, obtained independently in \cite{B, KW} 
are used to test the accuracy of the numerical methods in the case 
of a regular collision kernel $B\equiv 1$, see for example \cite{ZYY}).
Finally, we do hope that our work will give some clues to formulate new mathematical conjectures.

The paper is organized as follows. In section 2, we state the main theoretical results. 
The numerical details and algorithms are provided in section 3. Sections 4 and 5 
present the numerical results of the Boltzmann equation with different initial data 
for the Cauchy problem: 
we discuss in section 4 the results for a small $L^2$ initial data (bi-Gaussian); 
in section 5, we consider the case of a measure initial data. 
After that, we give a conclusion for this work. 
The paper ends with an appendix where we set some technical results.

\section{Theoretical results}

In this section, we present some theoretical parts: we begin by linearizing 
the Boltzmann equation and giving the spectral decomposition of this equation.  

\subsection{Linearization of the Boltzmann equation}
We remark that $Q(\mu,\mu)=0$
where the Gaussian function is defined by 
\begin{equation*}
\mu(v) = \frac{1}{(2 \pi )^{3/2}} \, e^{-\frac{|v|^2}{2}}
\end{equation*}
and it is a stationary solution of the Boltzmann equation.
We consider now a perturbation $g$ of the Gaussian. Then the solution $f$ 
of \eqref{eq Boltz f} can be written as
\begin{align*}
&f(t,v) = \mu(v) + \sqrt{\mu (v)} \, g(t,v),\\
&F(v) = \mu(v) + \sqrt{\mu (v)} \, G(v).
\end{align*}
It is easy to show that $g$ is a solution of the Cauchy problem  
\begin{equation} \label{eq Boltz g}
\left\{ \begin{aligned}
         &\partial_t g+\mathcal{L}(g)={\bf \Gamma}(g, g),\,\\
         &g|_{t=0}=g(0,v)=G(v)
\end{aligned} \right.
\end{equation}
where
\begin{align*}
&\mathcal{L}(g) = -\frac{1}{\sqrt{\mu}}[{\bf {Q}}(\sqrt{\mu}g,\mu)
                +{\bf {Q}}(\mu,\sqrt{\mu}g)]  
\end{align*}
is a linear operator and
\begin{align*}
&{\bf {\Gamma}}(g, h)   = \frac{1}{\sqrt{\mu}}{\bf {Q}}(\sqrt{\mu}g,\sqrt{\mu}h)
\end{align*}
is a nonlinear operator.
We decompose the solution of \eqref{eq Boltz g} into a linear and nonlinear part:
\[
g(t,v) = 
\underbrace{e^{- t \mathcal{L}} \, G(v)}_{\text{linear part}} + 
\underbrace{e^{- t \mathcal{L}} \, h(t,v)}_{\text{nonlinear part}} 
\]
where 
$e^{\alpha \mathcal{L}}$ is the exponential of the linear operator 
defined by his spectral decomposition (see below) and 
the new function $h(t,v)$ satisfies the following equation
\begin{equation} \label{eq Boltz h}
\left\{\begin{array}{ll}
  & \displaystyle \partial_t h
    = e^{t \mathcal{L}}  \, 
      {\bf {\Gamma}}(e^{-t \mathcal{L}} \,(G+h),e^{-t \mathcal{L}} \,(G+h)),\\
  & h(0,v) = 0. \end{array}
\right.
\end{equation}
The linearized operator $\mathcal{L}$ is a positive unbounded symmetric operator 
on $L^2 (\mathbb{R}^3_v)$ (see \cite{C, NYKC2, NYKC1, NYKC14}) with the kernel
\[
\mathcal{N}=\text{span}\left\{\sqrt{\mu},\,\sqrt{\mu}v_1,\,\sqrt{\mu}v_2,\,
\sqrt{\mu}v_3,\,\sqrt{\mu}|v|^2\right\}.
\]
From a rescaling argument (see Appendix \ref{appendix rescaling}), 
we can always assume that the initial condition $G$ satisfies
\begin{equation*}
  G\in \mathcal{N}^\perp.
\end{equation*}
In \cite{ NYKC1}, For the radially symmetric case, the authors show that 
the linear Boltzmann operator behaves like the fractional harmonic oscillator 
$\mathcal{H}^s$ ($0 < s <1$) with
\[
  \mathcal{H}=-\Delta +\frac{|v|^2}{4}.
\]
We study in the next section the spectral properties of the operators 
$\mathcal{L}$ and ${\bf\Gamma}$.

\subsection{The spectral problem}
We introduce now an orthonormal basis of $L^2_r(\mathbb{R}^3)$
the radial symmetric functions of  $L^2(\mathbb{R}^3)$
involving the generalized Laguerre polynomials $L^{[\ell + \frac12]}_{n}$: 
for that, we set for any $n \geq 0$
\begin{equation}\label{phi=}
\varphi_{n}(v) = \left(\frac{n!}{\sqrt{2}\Gamma(n  + 3/2)}\right)^{1/2}
               e^{-\frac{|v|^{2}}{4}} L^{[\frac12]}_{n}\left(\frac{|v|^{2}}{2}\right) \, 
               \frac{1}{\sqrt{4\pi}} 
\end{equation}
where $\Gamma(\,\cdot\,)$ is the standard gamma function, for all $x>0$,
\[
\Gamma(x)=\int^{+\infty}_0t^{x-1}e^{-x}dx
\]
and the Laguerre polynomial $L^{(\alpha)}_{n}$~of order $\alpha$,~degree $n$ is
\begin{align*}
&L^{(\alpha)}_{n}(x)=\sum^{n}_{r=0}(-1)^{n-r}\frac{\Gamma(\alpha+n+1)}{r!(n-r)!
\Gamma(\alpha+n-r+1)}x^{n-r}.
\end{align*}
We have the spectral decomposition for the linear Boltzmann operator
\[
\mathcal{L}\, \varphi_{n} = \lambda_n\, \varphi_{n}
\quad\quad
n \geq 0, 
\]
with $\phi_0 = \sqrt{\mu}$, $\lambda_0 = 0$ and for $n \geq 1$ 
\begin{equation} \label{lambda=}
 \lambda_n = 2\, \int_0^{\frac{\pi}{4}} \beta(\theta) 
          \left(1-(\sin \theta)^{2n}-(\cos \theta)^{2n}\right) d\theta 
\end{equation}
where $\beta(\theta)$ is defined from the collision kernel (see \eqref{non-cutoff-intro})
\begin{equation}\label{non-cutoff}
  \beta(\theta) = \sin\theta \, b(\cos\theta) \approx 
           \frac{1}{|\theta|^{1+2s}}, \quad 0<s<1, \,\, 
           \theta\in \bigl(0,\frac{\pi}{2}\bigr].
\end{equation}
The two families $(\varphi_{n} (v))_{n\geq 0}$ 
and $(\lambda_n)_{n\geq 0}$ represent the eigenvectors and the eigenvalues of $\mathcal{L}$.
Remark that this diagonalization of the linearized Boltzmann operator 
with Maxwellian molecules is also verified in the cutoff case 
(see \cite{Bo, C, Dole, NYKC2, NYKC1}).

We consider the spectral expansion 
\begin{equation}\label{gn(t)= Gn=}
g(t,v) = \sum_{n=0}^{\infty} g_n(t) \, \varphi_{n}(v) , \quad
G(v) = \sum_{n=0}^{\infty} G_n \, \varphi_{n}(v) 
\end{equation}
where $g_n(t) = \Big(g(t,\cdot), \varphi_n(\cdot)\Big)_{L^2}$
and $G_n = \Big(G, \varphi_n\Big)_{L^2}$. 
By definition, we have
\begin{equation*}
  e^{-t \mathcal{L} } G(v) = \sum_{n=0}^\infty e^{-\lambda_n t} G_n \varphi_n (v).
\end{equation*}
It is the solution of the equation
\begin{equation*}
\left\{\begin{array}{ll}
  & \displaystyle \partial_t g^{\ell in} + \mathcal{L} \, g^{\ell in}
    = 0,\\
  & g^{\ell in}(0,v) = G(v). \end{array}
\right.
\end{equation*} 
Then the operator ${\bf {\Gamma}}$ satisfies
\begin{align*}
{\bf {\Gamma}}( \varphi_{p},  \varphi_{q}) &= \mu_{pq} \, \varphi_{p+q} 
\end{align*}
where the non-linear eigenvalues are given by 
\begin{equation} \label{mu=}
\mu_{pq} = 
    \left(
    \frac{(2p+2q+1)}{(2p+1)(2q+1)} \, C_{2p+2q}^{2p}  
    \right)^{\frac12} \, 
    \int_{|\theta|\leq \frac{\pi}{4}}
     \beta(\theta)\, (\sin\theta)^{2p}\, (\cos\theta)^{2q} d\theta
\end{equation}
for $p\geq 1, q\geq 0$ and 
\begin{equation*}
 \mu_{0q} = 
    - \int_{|\theta|\leq \frac{\pi}{4}} \beta(\theta)\, (1-(\cos\theta)^{2q}) d\theta
\end{equation*}
for $q\geq1$.
Following \cite{NYKC14}, we therefore derive from \eqref{eq Boltz g} the following 
infinite system of ordinary differential equations :
\begin{equation}\label{eq Boltz gn}
\left\{ \begin{aligned}
&g_0^{\prime}(t)=0,\,\quad g_1^{\prime}(t)=0,\,\\
&\text{for all}\,\, n\geq 2,\\
&g_n'(t) + \lambda_n \, g_n(t) 
  = \sum_{\substack{p+q=n \\ 0\leq p,\, q\leq n}} \mu_{pq}\, g_p(t)\, g_q(t)
\end{aligned} \right.
\end{equation}
with the initial conditions (see \eqref{gn(t)= Gn=}) 
\begin{equation*}
g_n(0)= G_n \quad \text{for} \quad n\geq 0.
\end{equation*}
The goal is to study the behavior of each function $t\to\displaystyle g_n(t)$. 

In the rest, we will focus on the computation and properties of this intermediate solution.
\begin{proposition}\label{prop formel}
We assume that $G\in \mathcal{N}^{\perp}$.
Then the intermediate solution $h(t,v)$ defined by \eqref{eq Boltz h} satisfies
\begin{equation}\label{h=sum hn phin}
h(t,v) = \sum_{n=0}^{\infty} h_n (t) \, \varphi_n (v)
\end{equation}
where $h_0 \equiv  h_1 \equiv h_2 \equiv h_3 \equiv 0$ 
and for all $n\geq 4$
\begin{equation}\label{hn=int}
h_{n}(t) = \sum_{ \substack{p+q = n \\ 2\leq p, \, q\leq n-2}} \int_{0}^{t}
          \mu_{pq}\, e^{-(\lambda_p + \lambda_q - \lambda_n)s}\, 
          \bigl(G_p + h_p(s)\bigr) \, \bigl(G_q + h_q(s)\bigr)\, ds .
\end{equation}
\end{proposition}
\begin{remark}
As we have seen before, we divide the function $g$  in two parts as follows:
\begin{equation}\label{g=gl+gnl}
g(t,v) = 
\underbrace{\sum_{n=0}^{\infty} e^{-\lambda_{n}\, t}\, G_{n}\, \varphi_{n}(v)}_{g^{\ell in}(t,v)} + 
\underbrace{\sum_{n=0}^{\infty} e^{-\lambda_{n}\, t}\, h_{n}(t)\, \varphi_{n}(v)}_{g^{n \ell}(t,v)},
\end{equation}
therefore the formal solution $f(t,v)$ can be written as
\begin{equation}\label{f=mu+mu2 G+h}
f(t,v) = \mu(v) + \sqrt{\mu(v)} \,
  \sum_{n=0}^\infty 
  \left( e^{-\lambda_n \, t} \, G_n + e^{-\lambda_n \, t} \, h_n(t) \right) 
  \, \varphi_n(v).
\end{equation}
\end{remark}
\begin{proof}[Proof of proposition \ref{prop formel}] :
As $G\in \mathcal{N}^{\perp}$, we get $G_0 = G_1 = 0$ and 
we can verify from \eqref{eq Boltz gn} that
\[ g_0(t) = g_1(t) = 0, \,\, g_2(t) = G_2\,e^{-\lambda_2\,t},  \,\, 
      g_3(t) = G_3\,e^{-\lambda_2\,t}\]
and therefore $h_0 \equiv  h_1 \equiv h_2 \equiv h_3 \equiv 0$. 
By \eqref{eq Boltz gn}, we may write
\begin{equation}\label{gn=Gn+hn}
g_n(t) = e^{-\lambda_n \, t}  G_n + e^{-\lambda_n \, t}  \, h_n(t)
\end{equation}
and
\begin{equation}\label{g_n'=sum qp qq}
g_n'(t) + \lambda_n \, g_n(t) 
= \sum_{\substack{p+q=n \\ 2\leq p,\, q\leq n-2}} \mu_{pq}\, g_p(t)\, g_q(t).
\end{equation}
We plug again the value of $g_n$ from \eqref{gn=Gn+hn} into the equation 
\eqref{g_n'=sum qp qq} and we get
\[ h_n'(t) = e^{\lambda_n \,t}\, 
  \sum_{\substack{p+q=n \\ 2\leq p,\, q\leq n-2}} \mu_{pq}\, g_p(t)\, g_q(t).\]
Note that $h_n(0) = g_n^{n\ell}(0)=0$.
Finally, plugging the expression of $g_p$ and $g_q$ from \eqref{gn=Gn+hn} 
into the previous equation and integrating we prove \eqref{hn=int}. 
Concerning the exact expression of the eigenvalue $\lambda_n$ and $\mu_{pq}$, 
see \cite{NYKC14}. This concludes the proof. 
\end{proof}
We introduce now the following notations. For a $k$-uplet $\alpha \in \mathbb{N}^k$,
\begin{align*}
\begin{split}
\Lambda_\alpha = \lambda_{\alpha_1} +\lambda_{\alpha_2} + \cdots + \lambda_{\alpha_k}, \\
G^\alpha = G_{\alpha_1} \times  G_{\alpha_2} \cdots \times  G_{\alpha_k}.
\end{split}
\end{align*}
\begin{proposition}\label{prop valeur h}
For each integer $n\geq4$, we define 
$I_n$ a set of admissible indices 
\begin{equation*}
I_n = \left\{ \alpha \in \mathbb{N}^k \,\Big|\, k\in \mathbb{N}^*,\,\,  
     \alpha_i\geq 2,\,\,  |\alpha|=n,\right\}.
\end{equation*}
Then for each  multi-index $\alpha, \beta, \in I_n$
there exists some real coefficients $c^\alpha_\beta$ 
which depends only on 
$\lambda_2,\ldots ,\lambda_n$ and
$\mu_{pq}$ for $2\leq p , q \leq n-2$, $p+q\leq n$
such that
\begin{equation}\label{hn = poly}
h_n(t) = 
  \sum_{\alpha,\beta\in I_n}
  c^\alpha_\beta \,\, G^\alpha  \, \left(1 - e^{-(\Lambda_\beta-\lambda_n)\,t} \right).
\end{equation}
\end{proposition}
\begin{proof}
We compute directly from \eqref{hn=int}
\[ 
h_4(t) = c_{(2,2)}^{(2,2)} \, {G_2}^2 \, \left(1 - e^{-(\Lambda_{(2,2)}-\lambda_4)\,t} \right)\]
where
\[c_{(2,2)}^{(2,2)} = \frac{\mu_{22}}{(\Lambda_{(2,2)}-\lambda_4)} \]
and
\[ 
h_5(t) = 
  c_{(2,3)}^{(2,3)} \, G_2 G_3 \left(1 - e^{-(\Lambda_{(2,3)}-\lambda_5)\,t} \right)
+ c_{(3,2)}^{(3,2)} \, G_3 G_2 \left(1 - e^{-(\Lambda_{(3,2)}-\lambda_5)\,t} \right) \]
where
\[ 
c_{(2,3)}^{(2,3)} = \frac{\mu_{23}}{(\Lambda_{(2,3)}-\lambda_5)} \quad \text{and}\quad
c_{(3,2)}^{(3,2)} =\frac{\mu_{32}}{(\Lambda_{(3,2)}-\lambda_5)} .
 \]
We prove the result by induction.
Then we can suppose that \eqref{hn = poly} is true for each $h_{n'}$ ($4\leq n'\leq n-1$).
We will use the integral expression \eqref{hn=int} of $h_n$. 
We consider two integers $p, q$ such that $2 \leq p\, , q \leq n-2$ and $p+q=n$. 
Then from \eqref{hn = poly}
\begin{align*}
h_p(t) &= 
  \sum_{\alpha,\beta\in I_p}
  c^\alpha_\beta \,\, G^\alpha  \, \left(1 - e^{-(\Lambda_\beta-\lambda_p)\,t} \right), 
  \\
h_q(t) &= 
  \sum_{\alpha',\beta'\in I_q}
  c^{\alpha'}_{\beta'} \,\, G^{\alpha'}  \, \left(1 - e^{-(\Lambda_{\beta'}-\lambda_q)\,t} \right).
\end{align*}
From the integral formula \eqref{hn=int}
we get
\[ h_n(t) = \int_0^t \sum_{\substack{ p+q = n \\ 2\leq p, \, q\leq n-2}} (A+B+C+D) \, ds \]
with
\begin{align*}
& A = \mu_{pq} \, G_p\, G_q \, e^{-(\lambda_p+\lambda_q-\lambda_n)\,s},  \, 
\\
& B = \sum_{\alpha',\beta'\in I_q}
   \mu_{pq} \, c^{\alpha'}_{\beta'} \,\, G_p\,G^{\alpha'}  \, 
   (e^{-(\lambda_p+\lambda_q-\lambda_n)\,s} - e^{-(\lambda_p+\Lambda_{\beta'}  - \lambda_n)\,s}) , 
\\  
& C =   \sum_{\alpha,\beta\in I_p}
   \mu_{pq} \, c^{\alpha}_{\beta} \,\,G^{\alpha} \, G_q \, 
   (e^{-(\lambda_p+\lambda_q-\lambda_n)\,s} - e^{-(\Lambda_{\beta} + \lambda_q  - \lambda_n)\,s}) , 
\\ 
& D = 
\sum_{\alpha,\beta\in I_p} \,\,
\sum_{\alpha',\beta'\in I_q}
     \mu_{pq} \,  \, c^{\alpha}_{\beta} \, \, c^{\alpha'}_{\beta'} \, \, G^{\alpha} \, G^{\alpha'} 
     \quad  
\times
      \\
&  (e^{-(\lambda_p+\lambda_q-\lambda_n)\,s} - e^{-(\Lambda_\beta + \lambda_q - \lambda_n)\,s} - 
      e^{-(\lambda_p + \Lambda_{\beta'} - \lambda_n)\,s}+ e^{-(\Lambda_{\beta}+\Lambda_{\beta'} - \lambda_n)\,s}). 
\end{align*}
Expanding each previous terms and integrating over $[0,t]$, 
we get the result \eqref{hn = poly} since each number
$\lambda_p+\lambda_q-\lambda_n$, 
$\Lambda_\beta + \lambda_q - \lambda_n$,
$\lambda_p + \Lambda_{\beta'} - \lambda_n$,
$\Lambda_{\beta} +\Lambda_{\beta'} - \lambda_n$
are positive from the next lemma and 
$|\alpha|=|\beta|=p$, $|\alpha'|=|\beta'|=q$ and $p+q=n$.
\end{proof}
\begin{lemma}
The linear eigenvalues $\lambda_n$ for the non-cutoff radially symmetric 
spatially homogeneous Boltzmann equation
\[
\lambda_n = \int_{|\theta|\le \frac{\pi}{4}}\,\,\beta(\theta)\, 
        \Big(1 - (\sin\theta)^{2n}\, -(\cos\theta)^{2n}\, \Big)d\theta, 
\quad\quad n\geq 2,
\]
verify the following property
\[
\lambda_{\alpha_1 + \cdots + \alpha_k} <
\lambda_{\alpha_1} +  \cdots + \lambda_{\alpha_k}
(=\Lambda_\alpha)
\]
for multi-index $\alpha\in (\mathbb{N}\setminus\{0,1\})^k$.
\end{lemma}
\begin{proof}
By \cite{NYKC14}, we may write
\[
\lambda_{\alpha_1 + \alpha_2} < \lambda_{\alpha_1} + \lambda_{\alpha_2} , 
\]
then by iteration, we have

$
\lambda_{(\alpha_1 + \cdots + \alpha_k) + \alpha_{k+1}} < 
\lambda_{\alpha_1 + \cdots + \alpha_k} + \lambda_{\alpha_{k+1}} 
< (\lambda_{\alpha_1} + \cdots + \lambda_{\alpha_k})+ \lambda_{\alpha_{k+1}}.$
\end{proof}
%

\section{Numerical computations}

From now on, for sake of simplicity, we consider the specific case $s=\frac12$ and 
$$\beta(\theta) = (\sin \theta)^{-2}.$$
For the general case  $s\in]0,1[$ and other kernel $\beta$ which satisfies \eqref{non-cutoff}, 
we can compute some numerical approximations of the eigenvalues. 
We think that the results do not change. 

\subsection{Computation of the eigenvalues } \label{comp lambda}
By the following assumption
$\beta(\theta) \underset{0}{\approx} \frac{1}{|\theta|^{2}}$, 
we obtain (see \cite{NYKC1})
\begin{equation} \label{lambda approx}
\lambda_n \underset{\infty}{\approx} \sqrt{n} 
\end{equation}
where the linear eigenvalues $\lambda_n$ of $\mathcal{L}$ 
was defined in \eqref{lambda=}.
We recall the value of $\lambda_n$ for $n\geq 2$:
\begin{align*}
& \lambda_n = 2\, \int_0^{\frac{\pi}{4}} \beta(\theta) 
  \left(1-(\sin \theta)^{2n}-(\cos \theta)^{2n}\right) d\theta. 
\end{align*}  
We compute the exact and approximate values of $\lambda_n$ by the following algorithm : 
\smallskip
\textit{
\null\\
\null\hskip1cm$\lambda_0  \leftarrow 0$
\\
\null\hskip1cm for $n$  from 1 to N  do\\ 
\smallskip
\null\hskip2cm   
  $\text{expr} \leftarrow \text{algebraic simplification of } \,\,
   \frac{1-\sin^{2n}\theta-\cos^{2n}\theta}{\sin^2\theta}$
 \\\smallskip
\null\hskip2cm  $\lambda^{\text{exact}}_n $
  $\displaystyle \leftarrow
   \text{symbolic computation of }\,\, 
   2 \int_0^{\frac{\pi}{4}} \text{expr} \, d\theta$\\\smallskip
\null\hskip2cm  
  $\lambda^{\text{approx}}_n
    \leftarrow \text{numerical computation of}  \,\,\,  \lambda^{\text{exact}}_n$
}

\smallskip
\noindent
The ``algebraic simplification'' of ``expr'' 
removes the singularity when $\theta \to 0$ 
coming from the collision kernel $\beta(\theta) = \sin^{-2}\theta$ (see \eqref{non-cutoff}).
It consists in a factorization of trigonometric polynomials.
The symbolic computation of $\lambda^{\text{exact}}_n$ is reduced to compute 
the exact integral of a trigonometric polynomial.
Then $\lambda^{\text{exact}}_n$ is approached numerically with a number of significant digits
(equal to 10 in \ref{tabl lambda}).
The ap\-pro\-ximation is easily controlled by the estimate of the relative error
$|\lambda^{\text{exact}}_n - \lambda^{\text{appr.}}_n| / \lambda^{\text{exact}}_n$.
Using the software Maple$^{\mbox{\scriptsize{\textregistered}}}$13, 
we finally get the numerical table \ref{tabl lambda}. 

\begin{table}[h]
\begin{tabular}{|l|c|c|c|}
  \hline
   & Exact value & Approximate value 
   & Relative error  \\
  \hline
  \rule[-1.5ex]{0pt}{4ex} $\lambda_{1}$ & 0 & 0 & -- \\
  \hline 
  \rule[-1.5ex]{0pt}{4ex} $\lambda_{2}$ & ${1+ \frac12\,\pi}$ & 2.570796327 
        & $8.0\times10^{-11}$ \\
  \hline
  \rule[-1.5ex]{0pt}{4ex} $\lambda_{3}$ & $\frac32 + \frac34 \,\pi$ & 3.856194490 
       & $5.0\times10^{-11}$  \\
  \hline
  \rule[-1.5ex]{0pt}{4ex} $\lambda_{4}$ & ${\frac {23}{12}}+{\frac {15}{16}}\,\pi$ 
       & 4.861909780 & $1.2\times10^{-10}$  \\
  \hline
  \rule[-1.5ex]{0pt}{4ex} $\lambda_{5}$ & ${\frac {55}{24}}+{\frac {35}{32}}\,\pi$ 
       & 5.727783632 & $8.2\times10^{-11}$  \\
  \hline
  \rule[-1.5ex]{0pt}{4ex} $\lambda_{10}$ & ${\frac {61717}{16128}}+{\frac {109395}{65536}}\,\pi$ & 9.070756042 & $9.0\times10^{-11}$  \\
  \hline
  \rule[-1.5ex]{0pt}{4ex} $\lambda_{15}$ & ${\frac {41349267}{8200192}}+{\frac {35102025}{16777216}}\,\pi$ 
       & 11.61545300 & $3.2\times10^{-10}$  \\
  \hline
  \rule[-1.5ex]{0pt}{4ex} $\lambda_{20}$ & ${\frac {60225247403}{9906683904}}+{\frac {83945001525}{34359738368}}\,\pi$ 
       & 13.75454524 & $2.5\times10^{-11}$  \\
  \hline
\end{tabular}
\vskip8pt
\caption{Symbolic and numerical computation of $\lambda_n$.} \label{tabl lambda}
\end{table}
The approximation of eigenvalues can be controlled to be sufficiently precise 
for upcoming computations. 
For a general kernel $\beta(\theta)$, there is in general no more explicit values. 
But some classical numerical me\-thods can be easily applied.
Nevertheless, there is no more any algebraic simplification, 
and it is necessary to treat carefully the singularity.  
 
\subsection{Computation of the nonlinear eigenvalues}
 
We recall the coefficients $\mu_{pq}$ from \eqref{mu=}: 
for some $p,q\geq 1$
\begin{align*}
& \mu_{pq} = 
    \sqrt{\frac{(2p+2q+1)}{(2p+1)(2q+1)} \, C_{2p+2q}^{2p}} \, 
    \int_{|\theta|\leq \frac{\pi}{4}} \beta(\theta)\, (\sin\theta)^{2p}\, (\cos\theta)^{2q} d\theta. 
\end{align*}
We compute the exact and the approximate value of $\mu_{pq}$ 
(again with a relative error $\approx 10^{-10}$)
for $1\leq p+q \leq N$ by the following algorithm : 
\smallskip
\textit{
\null\\
\null\hskip0.66cm for $p$  from  1 to  N  do\\ 
\null\hskip1.32cm for  $q$   from  0   to  N-p   do\\ 
\smallskip
\null\hskip2cm  $\text{expr}$ 
  $\leftarrow  \displaystyle 
   \text{symbolic computation of }\,\, 
   2 \int_0^{\frac{\pi}{4}} \sin^{2p-2}\theta\,\cos^{2q}\theta  \, d\theta$\\\smallskip
\null\hskip2cm  $\mu^{\text{exact}}_{pq}$ 
  $\leftarrow 
     \sqrt{\frac{(2p+2q+1)}{(2p+1)(2q+1)} \, C_{2p+2q}^{2p}} \,\,\times \,\, 
     \text{expr}$
  \\\smallskip
\null\hskip2cm  
  $\mu^{\text{approx}}_{pq}
    \leftarrow \text{numerical computation of}  \,\,
      \mu^{\text{exact}}_{pq}$ 
}

\smallskip\noindent
We present in the table \ref{tabl mu} of results for $p+q=n=2,\ldots,5, 20$.
\begin{table}[h]
\begin{tabular}{|c|c|c|c|c|}
  \hline
  \rule[-1.5ex]{0pt}{4ex}  $n=2$ & $n=3$ & $n=4$ & $n=5$  &$n=20$\\
  \hline
  \rule[-1.5ex]{0pt}{4ex}   $\mu_{1,1} \approx 2.35$ & $\mu_{1,2} \approx 2.88$ 
  & $\mu_{1,3} \approx 3.29$ & $\mu_{1,4} \approx 3.62$
  &  $\mu_{1,19} \approx 6.68$ \\
  \hline
  \rule[-1.5ex]{0pt}{4ex}  & $\mu_{2,1} \approx 0.519$ & $\mu_{2,2} \approx 0.702$ & 
    $\mu_{2,3} \approx 0.84$ 
    &  $\mu_{2,18} \approx 1.55$ \\
  \hline
  \rule[-1.5ex]{0pt}{4ex}   &  & $\mu_{3,1} \approx 0.196$ & $\mu_{3,2} \approx 0.30$  
  &  $\mu_{3,17} \approx 0.75$ \\
  \hline
  \rule[-1.5ex]{0pt}{4ex}   & &  &  $\mu_{4,1} \approx 0.084$ 
  &  $\mu_{4,16} \approx 0.46$ \\
  \hline
  \rule[-1.5ex]{0pt}{4ex}   & &  &   
  &  $\vdots$ \\
  \hline  
  \rule[-1.5ex]{0pt}{4ex}   & &  &   
  &  $\mu_{19,1} \approx  10^{-5}$ \\  
  \hline  
\end{tabular}  
\vskip8pt
\caption{Numerical computation of $\mu_{pq}$.} 
\label{tabl mu}
\end{table}
The singularity coming from the collision kernel $\beta(\theta)=\sin^{-2}\theta$ 
is removed by a simple simplification (remark the exponent $(2p-2)$ of the 
sinus term of $\mu_{pq}$). 
Again for a general collision kernel, the values of these nonlinear eigenvalues 
can be approximated by classical numerical methods.

\subsection{Numerical solutions of the linear problem}

We introduce from \eqref{g=gl+gnl} the approximation of the linear solution  
\begin{equation}\label{gln N=}
g^{\ell in}_N(t,v) = \sum_{n=0}^{N} e^{-\lambda_n \, t} \, G_n  \, \varphi_n (v) 
\end{equation}
where the reals $G_n$ are the given initial spectral coefficients.
In order to compute the value of the linear solution, 
we use the formula \eqref{phi=} of the eigenfunction $\varphi_n$ 
which involves the generalized Laguerre polynomials $L^{[\ell + \frac12]}_{n}$.
We get the following algorithm : 
\smallskip
\textit{
\null\\
\null\hskip1cm for $n$  from 0 to N  do
\\
\null\hskip2cm  $\varphi_n(v) \leftarrow 
 \left(\frac{n!}{\sqrt{2}\Gamma(n  + 3/2)}\right)^{1/2}e^{-\frac{|v|^{2}}{4}}
L^{[\frac12]}_{n}\left(\frac{|v|^{2}}{2}\right)  \, \frac{1}{\sqrt{4\pi}}$
}

\smallskip
\noindent
Finally, we obtain the linear solution by the sum : 
\smallskip
\textit{
\null\\
\null\hskip1cm$g^{\ell in}_N(t,v)  \leftarrow 0$ 
\\
\null\hskip1cm for $n$ from 2 to N do\\
\null\hskip2cm  $g^{\ell in}_N(t,v) 
\leftarrow  g^{\ell in}_{N} (t,v) 
  + e^{-\lambda_n t} \, G_n \, \varphi_n(v) $ 
}

\smallskip
\noindent
We estimate the $L^2$ theoretical error $(g^{\ell in} -  g^{\ell in}_N)$
for the different initial data $G$ used for computation in the next sections. 

\begin{proposition}
We consider the solution of the following linear problem 
\begin{equation}\label{eq gnl}
\left\{\begin{array}{ll}
  & \displaystyle \partial_t g^{\ell in} + \mathcal{L} \, g^{\ell in}
    = 0,\\
  & g^{\ell in}(0,v) = G(v). \end{array}
\right.
\end{equation}
We have the following estimates :  \\
1) For initial data $G \in L^2$,
\[
\|g^{\ell in}(t,\cdot)  - g^{\ell in}_N (t,\cdot)  \|_{L^2}
\lesssim e^{-c \, \sqrt{N}\, t}\, \| G \|_{L^2}. 
\]
2) For the measure initial data $G$ defined by \eqref{G measure}
(see also proposition \ref{prop F=mu+d}), 
there exist some constants $C>0$ and $c>0$ such that  
for $t>0$ 
\[
\|g^{\ell in}(t,\cdot)  - g^{\ell in}_N (t,\cdot)  \|_{L^2}  
\lesssim \frac{1}{t^b}\, e^{-\gamma\,\sqrt{N}\,t} .\] 
\end{proposition}
\begin{proof}
The solution of \eqref{eq gnl} is 
\[ 
g^{\ell in}(t,v) = \sum_{n=0}^{\infty}  e^{-\lambda_n \, t} \, G_n  \, \varphi_n (v). 
\]
The exact error in $L^2$ is
\[
\|g^{\ell in}(t,\cdot)  - g^{\ell in}_N (t,\cdot)  \|_{L^2}^2 
= \sum_{n=N+1}^{\infty} e^{-2\lambda_n \, t} \, |G_n|^2 .
\]
1) If $G\in L^2(\mathbb{R}_{v}^{3})$, then as we have from \eqref{lambda approx}
\[
\|g^{\ell in}(t,\cdot)  - g^{\ell in}_N (t,\cdot)  \|_{L^2}^2 
= \sum_{n=N+1}^{\infty} e^{-2\lambda_n \, t} \, |G_n|^2 
\lesssim e^{-2\,c \, \sqrt{N}\, t}\, \| G \|_{L^2}^2.\]
We can deduce that
the exact error tends to zero when $N$ tends to infinity.

2) We suppose now that F is the measure initial data $\mu+\delta$. 
We can approximate the spectral coefficients $G_n$ of G by $n^{\frac14}$ 
and by \ref{comp lambda} we can then find some positive constants $c$ and $C$ such that
\[ \|g^{\ell in}(t,\cdot)  - g^{\ell in}_N (t,\cdot)  \|_{L^2}^2 
\leq C \, \sum_{n=N+1}^{\infty} e^{-c \, \sqrt{n}  \, t} \, n^2 .\]
We consider the function $\rho_t$ defined on $\mathbb{R}_{+}$ by 
$\rho_t(x) = e^{-c \, \sqrt{x}  \, t} \, x^2$.
So that $\rho_t$ is positive, continuous and decreasing for $x\geq 16/(c\,t)^2$, 
therefore by using the Cauchy integral criterion, we can write the following inequality~:
\[ \|g^{\ell in}(t,\cdot)  - g^{\ell in}_N (t,\cdot)  \|_{L^2} 
\leq 
\frac{C}{t^b}\, e^{-\gamma\,\sqrt{N}\,t}\underset{N\rightarrow \infty}{\longrightarrow 0} \]
where $b$ and $\gamma$ are some positive constants.  
\end{proof}

\subsection{Numerical solutions of the non-linear part}\label{meth.int}

Concerning the nonlinear part 
$g^{n\ell} = e^{-t \mathcal{L}} h$ of the solution, 
we consider the partial series 
\begin{equation}\label{gnl N=}
g^{n\ell}_N(t,v) = \sum_{n=0}^{N} e^{-\lambda_n \, t} \, h_n(t)  \, \varphi_n (v).
\end{equation}
We then use the decomposition of $h$ 
in the spectral basis \eqref{h=sum hn phin} 
and the integral formula \eqref{hn=int} to compute $h_n(t)$.  
Therefore we solve explicitly the system \eqref{eq Boltz gn} by the following algorithm:
\smallskip
\textit{
\null\\
\null\hskip1cm $h_0(t),\,\, h_1(t),\,\,  h_2(t),\,\,  h_3(t) \,\,  \leftarrow  \,\, 0$
\\
\null\hskip1cm for $n$  from 4 to N  do\\\smallskip
\null\hskip2cm  $S \leftarrow 0$\\\smallskip
\null\hskip2cm   for $p$ from 4 to n do\\\smallskip
\null\hskip3cm  $q \leftarrow n-p$\\\smallskip
\null\hskip3cm  $S \leftarrow S + 
  \mu_{pq} \, (G_p+h_p(t)) \, (G_q+h_q(t)) \, 
   e^{-(\lambda_p+\lambda_q-\lambda_n) \, t}$\\\smallskip
\null\hskip2cm  $h_n(t) \leftarrow 
  \text{symbolic computation of }\,\,
   \displaystyle  \int_0^t S $ 
}

\smallskip
\noindent
The exact computation of the integral $\int_0^t S $ 
is straightforward since, from proposition \ref{prop valeur h},
the symbolic expression $S$
is an linear combination of exponential terms $e^{\alpha t}$.
We get the exact following solutions of the system of integral formula \eqref{hn=int}: 
\begin{align*}
h_{0} =& \, h_{1} = h_{2} = h_{3} =  0,
\\
h_{{4}} = & \frac{\mu_{22}}{\lambda_2+\lambda_2-\lambda_4}\,
  {G_{{2}}}^{2} \left( 1- {e^{-(\lambda_2+\lambda_2-\lambda_4)\,t}}
 \right),  \\ 
h_5 = & 
 \frac{\mu_{23} + \mu_{32}}{\lambda_2+\lambda_3-\lambda_5} \, 
 G_2 G_3 \left(1 - e^{-(\lambda_2+\lambda_3-\lambda_5)\,t}
 \right),  \\
\cdots
\end{align*}
From the symbolic expression of $h_n$ we compute the numerical approximation :
\begin{align*}
h_{0} =& \, h_{1} = h_{2} = h_{3} =  0, 
\\
h_{{4}} = & 2.51\,{G_{{2}}}^{2} \left( 1- {e^{- 0.279\,t}}
 \right), \\ 
h_{{5}} = & 1.62\,G_{{2}}G_{{3}} \left( 1- {e^{- 0.698\,t}}
 \right), \\ 
h_{{6}} = & 0.322\,{G_{{3}}}^{2} \left( 1- {e^{- 1.20\,t}} \right) +  
1.17\, \left( 1 - \,{ e^{- 0.928\,t}} \right) G_{{2}}G_{{4}}\\
& + \left( - 2.95\,{e^{- 0.928\,t}}+ 0.677+ 2.26\,{
e^{- 1.20\,t}} \right) {G_{{2}}}^{3}, \\
h_{{7}} = & 0.501\,G_{{2}}G_{{5}} \left( 1- {e^{- 1.09\,t}}
 \right)+ 0.220\,\left( 1 - \,{e^{- 1.51\,t}} \right) G_{{3}
}G_{{4}}\\ 
& + \left(  0.201+ 0.478\,{e^{- 1.79\,t}}- 0.274\,{
e^{- 1.51\,t}}- 0.407\,{e^{- 1.09\,t}} \right) {G_{{2}}}^{
2}G_{{3}},
\\
...
\end{align*}
We finally get from \eqref{gnl N=} the approximation $g^{n\ell}_N$  
of the nonlinear part of the solution $g^{n\ell}(t,v)$ by 
the following algorithm : 
\smallskip
\textit{
\null\\
\null\hskip1cm $g^{n\ell}_N(t,v)  \leftarrow 0$ 
\\
\null\hskip1cm for $n$  from 2  to N  do\\
\null\hskip2cm  $g^{n\ell}_N(t,v) 
\leftarrow  g^{n\ell}_{N} (t,v) 
  + e^{-\lambda_n t} \, h_n(t) \, \varphi_n(v) $ 
}

\smallskip
\noindent
The symbolic and numerical computation of the nonlinear part of the solution 
plays the main difficulty of our method. 
We analyze the computation time and rounding off error in the next section.

\subsection{Discussions on the symbolic computation}
From the computation of the linear \eqref{gln N=} and nonlinear \eqref{gnl N=} part, 
we calculate the approximated solution of the Boltzmann equation \eqref{eq Boltz f}  
\begin{equation}\label{fN=}
f_N = \mu + \sqrt{\mu} \, (g_N^{\ell in} + g_N^{n\ell}).
\end{equation}
The method using the software Maple$^{\mbox{\scriptsize{\textregistered}}}$13 
and its internal function "{\bf int}$(f(x), x=a..b)$" for symbolic computation 
of integrals seems limited to a number $N$ around 20, since for $N=20$, 
the number of terms of $h_{20}$ is around 5000 
and the computation time is around 50 seconds.
Moreover they are both exponentially increasing (see Figure \ref{Fig time hn}).
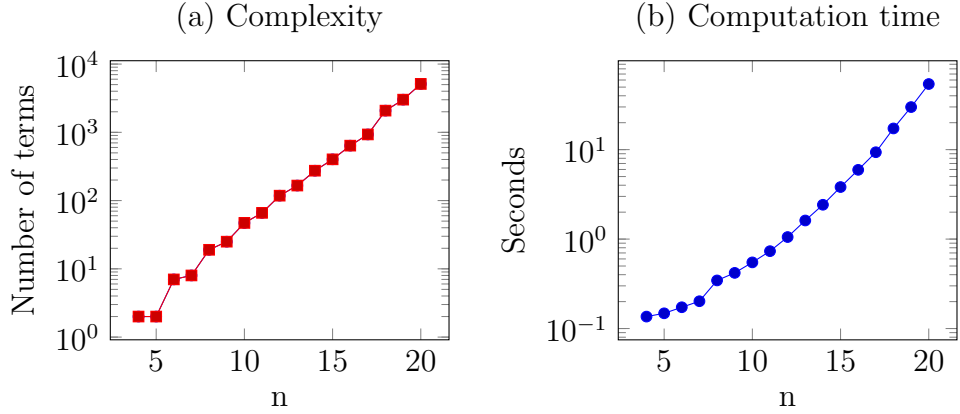
\begin{figure}[h]
\centering
\begin{tikzpicture}
\begin{semilogyaxis}
[scale=0.65,
title={(a) Complexity},
xlabel={n},
ylabel={Number of terms},
]
\addplot table {
4  2
5  2
6  7
7  8
8  19
9  25
10  47
11  66
12  118
13  166
14  274
15  402
16  638
17  931
18  2079
19  3011
20  5108
};%
\addplot coordinates { (4 , 2)  (5 , 2)  (6 , 7)  (7 , 8)  (8 , 19)  (9 , 25)  (10 , 47)  (11 , 66)  (12 , 118)  (13 , 166)  (14 , 274)  (15 , 402)  (16 , 638)  (17 , 931)  (18 , 2079)  (19 , 3011)  (20 , 5108)   }; 
\end{semilogyaxis}
\end{tikzpicture}
\quad
\begin{tikzpicture}
\begin{semilogyaxis}
[scale=0.65,
title={(b) Computation time},
xlabel={n},
ylabel={Seconds},
]
\addplot coordinates { (4 , 0.136)  (5 , 0.148)  (6 , 0.173)  (7 , 0.202)  (8 , 0.345)  (9 , 0.419)  (10 , 0.549)  (11 , 0.733)  (12 , 1.054)  (13 , 1.614)  (14 , 2.419)  (15 , 3.818)  (16 , 5.943)  (17 , 9.354)  (18 , 17.252)  (19 , 29.907)  (20 , 54.121)   };
\end{semilogyaxis}
\end{tikzpicture}
\caption{Number of terms for $h_n$ and computation time in seconds for $h_1,\ldots,h_n$.}
\label{Fig time hn}
\end{figure} 
We now estimate the truncation and rounding error due to the software computations.
For a regular $L^2$ initial data 
we have computed the solution $f_N$ for $N = 20$ and 
different number of digits 
(we can control the number of digits that Maple$^{\mbox{\scriptsize{\textregistered}}}$13 
uses when making calculations with software floating-point numbers). 
We set $P_1$ and $P_2$ two numbers of digits and
we compare the two numerical solutions $f_N^{P_1}$ and $f_N^{P_2}$ 
computed respectively using $P_1$ and $P_2$. 
We define the rounding relative error as   
\begin{equation*}
 \text{error} = \frac{\|f^{P_1}_N - f^{P_2}_N\|_{\infty}}{\|f^{P_2}_N\|_{\infty}} 
\end{equation*}
and we get the following results for different choices of $(P_1, P_2)$ :  
\begin{table}[h]
\begin{tabular}{|c|c|c|c|c|c|}
\hline 
$(P_1,P_2)$ & (10,20) & (20,30) & (30,40) & (40,50) & (50,100)
\\ 
\hline 
error & $3.8 \, \, 10^{-4}$ & $2.3\, \, 10^{-15}$ 
  & $2.0\, \, 10^{-29}$ & $8.2\, \, 10^{-40}$ & $1.4 \, \, 10^{-49}$  
\\ 
\hline 
\end{tabular}
\vskip8pt 
\caption{Relative rounding off error.}
\label{table error}
\end{table}
We check from the table \ref{table error} that the relative error is roughly 
10 times the precision of the computation of $f_N^{P_1}$.
The figure \ref{fig CPU time} represents the computation time of the solution $f_N$ 
with a regular $L^2$ initial data and for different numbers of digits. 
%
\begin{figure}[h]
\centering
\begin{tikzpicture}
\begin{axis}
[
height=6cm,
width=12cm,
xscale=0.65,
yscale=0.65,
xlabel={Number of digits},
xtick distance=500,
ylabel={CPU Time in sec.}, 
ylabel shift={10pt}, 
]
\addplot coordinates {  
 (50 , 90.991)  
 (210, 100) (500,135) (1000,138) (1500,319) (2000,618) }; 
\end{axis}
\end{tikzpicture}
\caption{Computation time for different number of digits.}
\label{fig CPU time}
\end{figure}
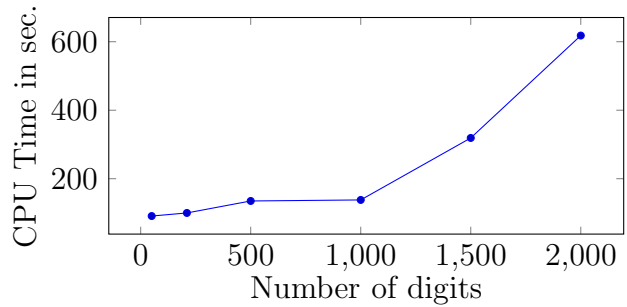
%
%
The computations of the solution $f_N$ 
was run on a computer having 8 Xeon processors 2.33 GHz with 8 GB of memory.
The method using Maple$^{\mbox{\scriptsize{\textregistered}}}$13 on this computer
seems limited to a number around $N=20$.   
Surprisingly, the computation time is roughly the same (around 90 seconds) 
for a number of digits between 20 and 1000. 
The main part of this time is therefore used for algebraic manipulation.

We present in the two upcoming sections the results of the computation 
for different initial values.


\section{Radial bi-Gaussian initial value}
We set the initial data : 
\begin{equation*}
\tilde F(w) =  
\frac{1}{(2\pi)^{\frac{3}{2}}}
\Big(
  \exp\big(-\frac12\, (|w|+1)^2 \big)
+ \exp\big(-\frac12\, (|w|-1)^2 \big) \Big).
\end{equation*}
We next rescale the initial data following lemma \ref{lem G ortho N}. 
We show in figure~\ref{Fig bi-gauss F1} the spectral approximation $F_N(v)$ 
of the initial data $F(v)$ such that $F_N(v)= \mu(v) + \sqrt{\mu(v)}\, G_N(v)$ 
where $G_N(v)=\sum_{n=0}^N G_n\,\varphi_n(v)$. 
%
\begin{figure}[h]
\centering
\begin{tikzpicture}
\begin{axis}
[scale=1.1,
xlabel={v},
scaled ticks=false,
legend entries={$F$,$F_{20}$,$F_{10}$,$F_{5}$},
],
\addplot[ 
mark=none, smooth, color=black, dashed, thick,
] 
coordinates { (-5 , 2.57709e-11)  (-4.875 , 9.90114e-11)  (-4.75 , 3.63738e-10)  (-4.625 , 1.27773e-09)  (-4.5 , 4.29177e-09)  (-4.375 , 1.37842e-08)  (-4.25 , 4.23325e-08)  (-4.125 , 1.24312e-07)  (-4 , 3.49060e-07)  (-3.875 , 9.37203e-07)  (-3.75 , 2.40611e-06)  (-3.625 , 5.90668e-06)  (-3.5 , 1.38650e-05)  (-3.375 , 3.11201e-05)  (-3.25 , 6.67900e-05)  (-3.125 , 0.000137066)  (-3 , 0.000268963)  (-2.875 , 0.000504667)  (-2.75 , 0.000905448)  (-2.625 , 0.00155335)  (-2.5 , 0.00254814)  (-2.375 , 0.0039969)  (-2.25 , 0.00599475)  (-2.125 , 0.00859739)  (-2 , 0.0117899)  (-1.875 , 0.0154596)  (-1.75 , 0.0193837)  (-1.625 , 0.0232394)  (-1.5 , 0.026642)  (-1.375 , 0.0292057)  (-1.25 , 0.0306158)  (-1.125 , 0.030693)  (-1 , 0.0294336)  (-0.875 , 0.0270129)  (-0.75 , 0.0237532)  (-0.625 , 0.0200654)  (-0.5 , 0.0163814)  (-0.375 , 0.0130958)  (-0.25 , 0.0105266)  (-0.125 , 0.0088983)  (0 , 0.00834136)  (0.125 , 0.0088983)  (0.25 , 0.0105266)  (0.375 , 0.0130958)  (0.5 , 0.0163814)  (0.625 , 0.0200654)  (0.75 , 0.0237532)  (0.875 , 0.0270129)  (1 , 0.0294336)  (1.125 , 0.030693)  (1.25 , 0.0306158)  (1.375 , 0.0292057)  (1.5 , 0.026642)  (1.625 , 0.0232394)  (1.75 , 0.0193837)  (1.875 , 0.0154596)  (2 , 0.0117899)  (2.125 , 0.00859739)  (2.25 , 0.00599475)  (2.375 , 0.0039969)  (2.5 , 0.00254814)  (2.625 , 0.00155335)  (2.75 , 0.000905448)  (2.875 , 0.000504667)  (3 , 0.000268963)  (3.125 , 0.000137066)  (3.25 , 6.67900e-05)  (3.375 , 3.11201e-05)  (3.5 , 1.38650e-05)  (3.625 , 5.90668e-06)  (3.75 , 2.40611e-06)  (3.875 , 9.37203e-07)  (4 , 3.49060e-07)  (4.125 , 1.24312e-07)  (4.25 , 4.23325e-08)  (4.375 , 1.37842e-08)  (4.5 , 4.29177e-09)  (4.625 , 1.27773e-09)  (4.75 , 3.63738e-10)  (4.875 , 9.90114e-11)  (5 , 2.57709e-11)   }; 
\addplot[ 
mark=none, smooth, very thick, color=blue,
] 
coordinates { (-5 , 1.44929e-09)  (-4.875 , 1.40447e-09)  (-4.75 , 1.82299e-10)  (-4.625 , -1.58104e-09)  (-4.5 , -1.20572e-09)  (-4.375 , 8.55059e-09)  (-4.25 , 4.27480e-08)  (-4.125 , 1.34854e-07)  (-4 , 3.68464e-07)  (-3.875 , 9.54491e-07)  (-3.75 , 2.40327e-06)  (-3.625 , 5.87073e-06)  (-3.5 , 1.38040e-05)  (-3.375 , 3.10723e-05)  (-3.25 , 6.68085e-05)  (-3.125 , 0.000137179)  (-3 , 0.000269132)  (-2.875 , 0.000504774)  (-2.75 , 0.000905359)  (-2.625 , 0.00155303)  (-2.5 , 0.00254773)  (-2.375 , 0.00399672)  (-2.25 , 0.00599508)  (-2.125 , 0.00859821)  (-2 , 0.0117907)  (-1.875 , 0.0154598)  (-1.75 , 0.0193828)  (-1.625 , 0.0232376)  (-1.5 , 0.0266405)  (-1.375 , 0.0292059)  (-1.25 , 0.0306184)  (-1.125 , 0.0306968)  (-1 , 0.0294358)  (-0.875 , 0.027011)  (-0.75 , 0.0237468)  (-0.625 , 0.0200576)  (-0.5 , 0.0163785)  (-0.375 , 0.0131046)  (-0.25 , 0.0105503)  (-0.125 , 0.00893453)  (0 , 0.00838251)  (0.125 , 0.00893453)  (0.25 , 0.0105503)  (0.375 , 0.0131046)  (0.5 , 0.0163785)  (0.625 , 0.0200576)  (0.75 , 0.0237468)  (0.875 , 0.027011)  (1 , 0.0294358)  (1.125 , 0.0306968)  (1.25 , 0.0306184)  (1.375 , 0.0292059)  (1.5 , 0.0266405)  (1.625 , 0.0232376)  (1.75 , 0.0193828)  (1.875 , 0.0154598)  (2 , 0.0117907)  (2.125 , 0.00859821)  (2.25 , 0.00599508)  (2.375 , 0.00399672)  (2.5 , 0.00254773)  (2.625 , 0.00155303)  (2.75 , 0.000905359)  (2.875 , 0.000504774)  (3 , 0.000269132)  (3.125 , 0.000137179)  (3.25 , 6.68085e-05)  (3.375 , 3.10723e-05)  (3.5 , 1.38040e-05)  (3.625 , 5.87073e-06)  (3.75 , 2.40327e-06)  (3.875 , 9.54491e-07)  (4 , 3.68464e-07)  (4.125 , 1.34854e-07)  (4.25 , 4.27480e-08)  (4.375 , 8.55059e-09)  (4.5 , -1.20572e-09)  (4.625 , -1.58104e-09)  (4.75 , 1.82299e-10)  (4.875 , 1.40447e-09)  (5 , 1.44929e-09)   }; 
\addplot[ 
mark=none, smooth, color=red,] 
coordinates { (-5 , 9.19704e-08)  (-4.875 , 1.26989e-07)  (-4.75 , 1.25545e-07)  (-4.625 , 5.16632e-08)  (-4.5 , -1.21828e-07)  (-4.375 , -3.85527e-07)  (-4.25 , -6.58535e-07)  (-4.125 , -7.57375e-07)  (-4 , -3.80818e-07)  (-3.875 , 8.91852e-07)  (-3.75 , 3.63283e-06)  (-3.625 , 8.74590e-06)  (-3.5 , 1.80107e-05)  (-3.375 , 3.52952e-05)  (-3.25 , 6.87580e-05)  (-3.125 , 0.000134229)  (-3 , 0.000259606)  (-2.875 , 0.000489538)  (-2.75 , 0.000888944)  (-2.625 , 0.00154332)  (-2.5 , 0.00255356)  (-2.375 , 0.00402356)  (-2.25 , 0.00604025)  (-2.125 , 0.00864815)  (-2 , 0.0118228)  (-1.875 , 0.0154496)  (-1.75 , 0.0193162)  (-1.625 , 0.0231222)  (-1.5 , 0.0265112)  (-1.375 , 0.0291194)  (-1.25 , 0.0306344)  (-1.125 , 0.0308513)  (-1 , 0.0297146)  (-0.875 , 0.0273353)  (-0.75 , 0.0239801)  (-0.625 , 0.0200357)  (-0.5 , 0.0159556)  (-0.375 , 0.0122016)  (-0.25 , 0.00918908)  (-0.125 , 0.00724364)  (0 , 0.00657165)  (0.125 , 0.00724364)  (0.25 , 0.00918908)  (0.375 , 0.0122016)  (0.5 , 0.0159556)  (0.625 , 0.0200357)  (0.75 , 0.0239801)  (0.875 , 0.0273353)  (1 , 0.0297146)  (1.125 , 0.0308513)  (1.25 , 0.0306344)  (1.375 , 0.0291194)  (1.5 , 0.0265112)  (1.625 , 0.0231222)  (1.75 , 0.0193162)  (1.875 , 0.0154496)  (2 , 0.0118228)  (2.125 , 0.00864815)  (2.25 , 0.00604025)  (2.375 , 0.00402356)  (2.5 , 0.00255356)  (2.625 , 0.00154332)  (2.75 , 0.000888944)  (2.875 , 0.000489538)  (3 , 0.000259606)  (3.125 , 0.000134229)  (3.25 , 6.87580e-05)  (3.375 , 3.52952e-05)  (3.5 , 1.80107e-05)  (3.625 , 8.74590e-06)  (3.75 , 3.63283e-06)  (3.875 , 8.91852e-07)  (4 , -3.80818e-07)  (4.125 , -7.57375e-07)  (4.25 , -6.58535e-07)  (4.375 , -3.85527e-07)  (4.5 , -1.21828e-07)  (4.625 , 5.16632e-08)  (4.75 , 1.25545e-07)  (4.875 , 1.26989e-07)  (5 , 9.19704e-08)   }; 
\addplot[ 
mark=none, smooth, color=green, dotted, very thick,]
 coordinates { (-5 , -2.19536e-07)  (-4.875 , -2.24300e-07)  (-4.75 , -1.45024e-07)  (-4.625 , 7.62584e-08)  (-4.5 , 4.97551e-07)  (-4.375 , 1.15093e-06)  (-4.25 , 2.00379e-06)  (-4.125 , 2.91778e-06)  (-4 , 3.63402e-06)  (-3.875 , 3.84109e-06)  (-3.75 , 3.41741e-06)  (-3.625 , 2.97585e-06)  (-3.5 , 4.86201e-06)  (-3.375 , 1.47483e-05)  (-3.25 , 4.38994e-05)  (-3.125 , 0.000112039)  (-3 , 0.000250515)  (-2.875 , 0.000505142)  (-2.75 , 0.000937789)  (-2.625 , 0.0016255)  (-2.5 , 0.00265586)  (-2.375 , 0.00411764)  (-2.25 , 0.0060863)  (-2.125 , 0.00860521)  (-2 , 0.0116647)  (-1.875 , 0.0151827)  (-1.75 , 0.0189914)  (-1.625 , 0.0228356)  (-1.5 , 0.0263863)  (-1.375 , 0.0292716)  (-1.25 , 0.0311241)  (-1.125 , 0.0316384)  (-1 , 0.0306309)  (-0.875 , 0.0280891)  (-0.75 , 0.0242003)  (-0.625 , 0.0193496)  (-0.5 , 0.0140837)  (-0.375 , 0.00904311)  (-0.25 , 0.00487183)  (-0.125 , 0.00212059)  (0 , 0.00115996)  (0.125 , 0.00212059)  (0.25 , 0.00487183)  (0.375 , 0.00904311)  (0.5 , 0.0140837)  (0.625 , 0.0193496)  (0.75 , 0.0242003)  (0.875 , 0.0280891)  (1 , 0.0306309)  (1.125 , 0.0316384)  (1.25 , 0.0311241)  (1.375 , 0.0292716)  (1.5 , 0.0263863)  (1.625 , 0.0228356)  (1.75 , 0.0189914)  (1.875 , 0.0151827)  (2 , 0.0116647)  (2.125 , 0.00860521)  (2.25 , 0.0060863)  (2.375 , 0.00411764)  (2.5 , 0.00265586)  (2.625 , 0.0016255)  (2.75 , 0.000937789)  (2.875 , 0.000505142)  (3 , 0.000250515)  (3.125 , 0.000112039)  (3.25 , 4.38994e-05)  (3.375 , 1.47483e-05)  (3.5 , 4.86201e-06)  (3.625 , 2.97585e-06)  (3.75 , 3.41741e-06)  (3.875 , 3.84109e-06)  (4 , 3.63402e-06)  (4.125 , 2.91778e-06)  (4.25 , 2.00379e-06)  (4.375 , 1.15093e-06)  (4.5 , 4.97551e-07)  (4.625 , 7.62584e-08)  (4.75 , -1.45024e-07)  (4.875 , -2.24300e-07)  (5 , -2.19536e-07)   }; 
\end{axis}
\end{tikzpicture}
\caption{Approximation of the initial data}
\label{Fig bi-gauss F1}
\end{figure}
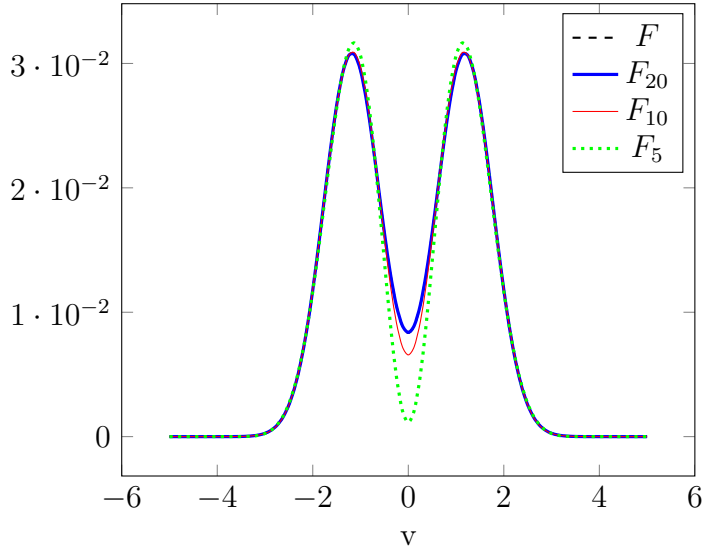
We then compute the solutions $h_n(t)$ from the proposition \ref{prop formel}
for $n = 4,5,\ldots,N$ with $N=20$. 
For each integer $n$, the function $t\to h_n(t)$ is monotone and 
tends to a finite limit when $t$ tends to infinity 
(See Figure \ref{Fig bi-Gauss CV} (a)).
We recall that $h_n(t)$ is a finite sum of decreasing exponential terms 
(see section \ref{meth.int}).
Since the initial data $G$ is a regular function, 
the spectral coefficients $G_n$ are exponentially decreasing. 
The numerical computation of $h_n(t)$ 
shows also that $\|h_n\|_{\infty}$ is exponentially decreasing with respect to $n$
(See Figure \ref{Fig bi-Gauss CV} (b)). 
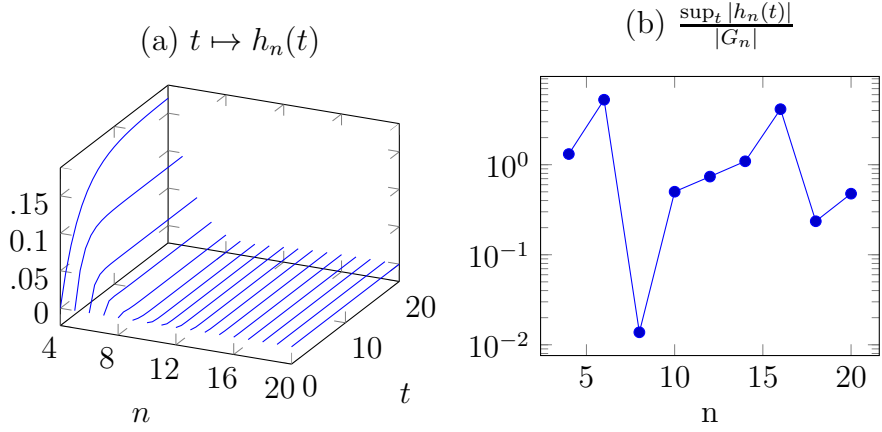
\begin{figure}[h]
\centering
\begin{tikzpicture}
\begin{axis}
[scale=0.65,
title= {(a) $t \mapsto h_n(t)$},
xlabel={$n$},
ylabel={$t$},
scaled ticks=false,
xtick={4,8,12,16,20},
ztick={0,0.05,0.1,0.15}, 
zticklabels={$0$,$.05$,$0.1$,$.15$}, 
]
\addplot3 table [mark=none, ]{
4  	0.000000e+00  	0 
4  	1.000000e+00  	0.0416132 
4  	2.000000e+00  	0.0730738 
4  	3.000000e+00  	0.0968587 
4  	4.000000e+00  	0.114841 
4  	5.000000e+00  	0.128435 
4  	6.000000e+00  	0.138713 
4  	7.000000e+00  	0.146484 
4  	8.000000e+00  	0.152358 
4  	9.000000e+00  	0.1568 
4  	1.000000e+01  	0.160157 
4  	1.100000e+01  	0.162696 
4  	1.200000e+01  	0.164615 
4  	1.300000e+01  	0.166066 
4  	1.400000e+01  	0.167163 
4  	1.500000e+01  	0.167992 
4  	1.600000e+01  	0.168619 
4  	1.700000e+01  	0.169093 
4  	1.800000e+01  	0.169452 
4  	1.900000e+01  	0.169723 
4  	2.000000e+01  	0.169928

5  	0.000000e+00  	0 
5  	1.000000e+00  	0.0481219 
5  	2.000000e+00  	0.0720375 
5  	3.000000e+00  	0.0839231 
5  	4.000000e+00  	0.0898299 
5  	5.000000e+00  	0.0927655 
5  	6.000000e+00  	0.0942245 
5  	7.000000e+00  	0.0949495 
5  	8.000000e+00  	0.0953099 
5  	9.000000e+00  	0.0954889 
5  	1.000000e+01  	0.0955779 
5  	1.100000e+01  	0.0956222 
5  	1.200000e+01  	0.0956442 
5  	1.300000e+01  	0.0956551 
5  	1.400000e+01  	0.0956605 
5  	1.500000e+01  	0.0956632 
5  	1.600000e+01  	0.0956645 
5  	1.700000e+01  	0.0956652 
5  	1.800000e+01  	0.0956655 
5  	1.900000e+01  	0.0956657 
5  	2.000000e+01  	0.0956658

6  	0.000000e+00  	0 
6  	1.000000e+00  	0.0318772 
6  	2.000000e+00  	0.0406701 
6  	3.000000e+00  	0.04301 
6  	4.000000e+00  	0.0435961 
6  	5.000000e+00  	0.0437267 
6  	6.000000e+00  	0.0437482 
6  	7.000000e+00  	0.0437478 
6  	8.000000e+00  	0.0437449 
6  	9.000000e+00  	0.043743 
6  	1.000000e+01  	0.043742 
6  	1.100000e+01  	0.0437415 
6  	1.200000e+01  	0.0437413 
6  	1.300000e+01  	0.0437412 
6  	1.400000e+01  	0.0437412 
6  	1.500000e+01  	0.0437412 
6  	1.600000e+01  	0.0437412 
6  	1.700000e+01  	0.0437411 
6  	1.800000e+01  	0.0437411 
6  	1.900000e+01  	0.0437411 
6  	2.000000e+01  	0.0437411

7  	0.000000e+00  	1e-11 
7  	1.000000e+00  	0.0139551 
7  	2.000000e+00  	0.0148689 
7  	3.000000e+00  	0.0145599 
7  	4.000000e+00  	0.0143564 
7  	5.000000e+00  	0.014272 
7  	6.000000e+00  	0.0142411 
7  	7.000000e+00  	0.0142303 
7  	8.000000e+00  	0.0142266 
7  	9.000000e+00  	0.0142253 
7  	1.000000e+01  	0.0142249 
7  	1.100000e+01  	0.0142248 
7  	1.200000e+01  	0.0142247 
7  	1.300000e+01  	0.0142247 
7  	1.400000e+01  	0.0142247 
7  	1.500000e+01  	0.0142247 
7  	1.600000e+01  	0.0142247 
7  	1.700000e+01  	0.0142247 
7  	1.800000e+01  	0.0142247 
7  	1.900000e+01  	0.0142247 
7  	2.000000e+01  	0.0142247

8  	0.000000e+00  	-2e-12 
8  	1.000000e+00  	0.0024101 
8  	2.000000e+00  	0.0010307 
8  	3.000000e+00  	0.000472908 
8  	4.000000e+00  	0.000299719 
8  	5.000000e+00  	0.000248389 
8  	6.000000e+00  	0.000233237 
8  	7.000000e+00  	0.000228745 
8  	8.000000e+00  	0.000227408 
8  	9.000000e+00  	0.000227008 
8  	1.000000e+01  	0.000226889 
8  	1.100000e+01  	0.000226853 
8  	1.200000e+01  	0.000226842 
8  	1.300000e+01  	0.000226839 
8  	1.400000e+01  	0.000226838 
8  	1.500000e+01  	0.000226838 
8  	1.600000e+01  	0.000226837 
8  	1.700000e+01  	0.000226837 
8  	1.800000e+01  	0.000226837 
8  	1.900000e+01  	0.000226837 
8  	2.000000e+01  	0.000226837

9  	0.000000e+00  	5e-12 
9  	1.000000e+00  	-0.00286386 
9  	2.000000e+00  	-0.00432937 
9  	3.000000e+00  	-0.00471077 
9  	4.000000e+00  	-0.00480807 
9  	5.000000e+00  	-0.00483356 
9  	6.000000e+00  	-0.00484037 
9  	7.000000e+00  	-0.00484221 
9  	8.000000e+00  	-0.00484271 
9  	9.000000e+00  	-0.00484285 
9  	1.000000e+01  	-0.00484289 
9  	1.100000e+01  	-0.0048429 
9  	1.200000e+01  	-0.0048429 
9  	1.300000e+01  	-0.0048429 
9  	1.400000e+01  	-0.0048429 
9  	1.500000e+01  	-0.0048429 
9  	1.600000e+01  	-0.0048429 
9  	1.700000e+01  	-0.0048429 
9  	1.800000e+01  	-0.0048429 
9  	1.900000e+01  	-0.0048429 
9  	2.000000e+01  	-0.0048429

10  	0.000000e+00  	3e-12 
10  	1.000000e+00  	-0.00423512 
10  	2.000000e+00  	-0.00526324 
10  	3.000000e+00  	-0.00548177 
10  	4.000000e+00  	-0.00553225 
10  	5.000000e+00  	-0.0055445 
10  	6.000000e+00  	-0.00554754 
10  	7.000000e+00  	-0.00554831 
10  	8.000000e+00  	-0.0055485 
10  	9.000000e+00  	-0.00554855 
10  	1.000000e+01  	-0.00554856 
10  	1.100000e+01  	-0.00554856 
10  	1.200000e+01  	-0.00554857 
10  	1.300000e+01  	-0.00554857 
10  	1.400000e+01  	-0.00554857 
10  	1.500000e+01  	-0.00554857 
10  	1.600000e+01  	-0.00554857 
10  	1.700000e+01  	-0.00554857 
10  	1.800000e+01  	-0.00554857 
10  	1.900000e+01  	-0.00554857 
10  	2.000000e+01  	-0.00554857

11  	0.000000e+00  	-2e-12 
11  	1.000000e+00  	-0.00378209 
11  	2.000000e+00  	-0.00439674 
11  	3.000000e+00  	-0.00451268 
11  	4.000000e+00  	-0.00453773 
11  	5.000000e+00  	-0.00454345 
11  	6.000000e+00  	-0.00454479 
11  	7.000000e+00  	-0.0045451 
11  	8.000000e+00  	-0.00454517 
11  	9.000000e+00  	-0.00454519 
11  	1.000000e+01  	-0.0045452 
11  	1.100000e+01  	-0.0045452 
11  	1.200000e+01  	-0.0045452 
11  	1.300000e+01  	-0.0045452 
11  	1.400000e+01  	-0.0045452 
11  	1.500000e+01  	-0.0045452 
11  	1.600000e+01  	-0.0045452 
11  	1.700000e+01  	-0.0045452 
11  	1.800000e+01  	-0.0045452 
11  	1.900000e+01  	-0.0045452 
11  	2.000000e+01  	-0.0045452

12  	0.000000e+00  	0 
12  	1.000000e+00  	-0.00275973 
12  	2.000000e+00  	-0.00309163 
12  	3.000000e+00  	-0.00314925 
12  	4.000000e+00  	-0.00316104 
12  	5.000000e+00  	-0.00316359 
12  	6.000000e+00  	-0.00316416 
12  	7.000000e+00  	-0.00316428 
12  	8.000000e+00  	-0.00316431 
12  	9.000000e+00  	-0.00316432 
12  	1.000000e+01  	-0.00316432 
12  	1.100000e+01  	-0.00316432 
12  	1.200000e+01  	-0.00316432 
12  	1.300000e+01  	-0.00316432 
12  	1.400000e+01  	-0.00316432 
12  	1.500000e+01  	-0.00316432 
12  	1.600000e+01  	-0.00316432 
12  	1.700000e+01  	-0.00316432 
12  	1.800000e+01  	-0.00316432 
12  	1.900000e+01  	-0.00316432 
12  	2.000000e+01  	-0.00316432

13  	0.000000e+00  	-1e-12 
13  	1.000000e+00  	-0.0017643 
13  	2.000000e+00  	-0.00192676 
13  	3.000000e+00  	-0.00195303 
13  	4.000000e+00  	-0.00195814 
13  	5.000000e+00  	-0.00195919 
13  	6.000000e+00  	-0.00195941 
13  	7.000000e+00  	-0.00195946 
13  	8.000000e+00  	-0.00195947 
13  	9.000000e+00  	-0.00195947 
13  	1.000000e+01  	-0.00195947 
13  	1.100000e+01  	-0.00195947 
13  	1.200000e+01  	-0.00195947 
13  	1.300000e+01  	-0.00195947 
13  	1.400000e+01  	-0.00195947 
13  	1.500000e+01  	-0.00195947 
13  	1.600000e+01  	-0.00195947 
13  	1.700000e+01  	-0.00195947 
13  	1.800000e+01  	-0.00195947 
13  	1.900000e+01  	-0.00195947 
13  	2.000000e+01  	-0.00195947

14  	0.000000e+00  	1e-12 
14  	1.000000e+00  	-0.00100339 
14  	2.000000e+00  	-0.00107259 
14  	3.000000e+00  	-0.0010829 
14  	4.000000e+00  	-0.0010848 
14  	5.000000e+00  	-0.00108518 
14  	6.000000e+00  	-0.00108525 
14  	7.000000e+00  	-0.00108527 
14  	8.000000e+00  	-0.00108527 
14  	9.000000e+00  	-0.00108527 
14  	1.000000e+01  	-0.00108527 
14  	1.100000e+01  	-0.00108527 
14  	1.200000e+01  	-0.00108527 
14  	1.300000e+01  	-0.00108527 
14  	1.400000e+01  	-0.00108527 
14  	1.500000e+01  	-0.00108527 
14  	1.600000e+01  	-0.00108527 
14  	1.700000e+01  	-0.00108527 
14  	1.800000e+01  	-0.00108527 
14  	1.900000e+01  	-0.00108527 
14  	2.000000e+01  	-0.00108527

15  	0.000000e+00  	7e-13 
15  	1.000000e+00  	-0.000497244 
15  	2.000000e+00  	-0.000519095 
15  	3.000000e+00  	-0.00052185 
15  	4.000000e+00  	-0.000522315 
15  	5.000000e+00  	-0.000522401 
15  	6.000000e+00  	-0.000522418 
15  	7.000000e+00  	-0.000522421 
15  	8.000000e+00  	-0.000522422 
15  	9.000000e+00  	-0.000522422 
15  	1.000000e+01  	-0.000522422 
15  	1.100000e+01  	-0.000522422 
15  	1.200000e+01  	-0.000522422 
15  	1.300000e+01  	-0.000522422 
15  	1.400000e+01  	-0.000522422 
15  	1.500000e+01  	-0.000522422 
15  	1.600000e+01  	-0.000522422 
15  	1.700000e+01  	-0.000522422 
15  	1.800000e+01  	-0.000522422 
15  	1.900000e+01  	-0.000522422 
15  	2.000000e+01  	-0.000522422

16  	0.000000e+00  	5e-13 
16  	1.000000e+00  	-0.000195708 
16  	2.000000e+00  	-0.000195952 
16  	3.000000e+00  	-0.000195517 
16  	4.000000e+00  	-0.000195413 
16  	5.000000e+00  	-0.000195391 
16  	6.000000e+00  	-0.000195387 
16  	7.000000e+00  	-0.000195386 
16  	8.000000e+00  	-0.000195386 
16  	9.000000e+00  	-0.000195386 
16  	1.000000e+01  	-0.000195386 
16  	1.100000e+01  	-0.000195386 
16  	1.200000e+01  	-0.000195386 
16  	1.300000e+01  	-0.000195386 
16  	1.400000e+01  	-0.000195386 
16  	1.500000e+01  	-0.000195386 
16  	1.600000e+01  	-0.000195386 
16  	1.700000e+01  	-0.000195386 
16  	1.800000e+01  	-0.000195386 
16  	1.900000e+01  	-0.000195386 
16  	2.000000e+01  	-0.000195386

17  	0.000000e+00  	5e-14 
17  	1.000000e+00  	-3.54328e-05 
17  	2.000000e+00  	-2.75458e-05 
17  	3.000000e+00  	-2.60596e-05 
17  	4.000000e+00  	-2.57848e-05 
17  	5.000000e+00  	-2.57338e-05 
17  	6.000000e+00  	-2.57242e-05 
17  	7.000000e+00  	-2.57225e-05 
17  	8.000000e+00  	-2.57221e-05 
17  	9.000000e+00  	-2.57221e-05 
17  	1.000000e+01  	-2.57221e-05 
17  	1.100000e+01  	-2.57221e-05 
17  	1.200000e+01  	-2.57221e-05 
17  	1.300000e+01  	-2.57221e-05 
17  	1.400000e+01  	-2.57221e-05 
17  	1.500000e+01  	-2.57221e-05 
17  	1.600000e+01  	-2.57221e-05 
17  	1.700000e+01  	-2.57221e-05 
17  	1.800000e+01  	-2.57221e-05 
17  	1.900000e+01  	-2.57221e-05 
17  	2.000000e+01  	-2.57221e-05

18  	0.000000e+00  	1.2e-13 
18  	1.000000e+00  	3.73679e-05 
18  	2.000000e+00  	4.69057e-05 
18  	3.000000e+00  	4.84909e-05 
18  	4.000000e+00  	4.87692e-05 
18  	5.000000e+00  	4.88191e-05 
18  	6.000000e+00  	4.88282e-05 
18  	7.000000e+00  	4.88298e-05 
18  	8.000000e+00  	4.88301e-05 
18  	9.000000e+00  	4.88302e-05 
18  	1.000000e+01  	4.88302e-05 
18  	1.100000e+01  	4.88302e-05 
18  	1.200000e+01  	4.88302e-05 
18  	1.300000e+01  	4.88302e-05 
18  	1.400000e+01  	4.88302e-05 
18  	1.500000e+01  	4.88302e-05 
18  	1.600000e+01  	4.88302e-05 
18  	1.700000e+01  	4.88302e-05 
18  	1.800000e+01  	4.88302e-05 
18  	1.900000e+01  	4.88302e-05 
18  	2.000000e+01  	4.88302e-05

19  	0.000000e+00  	4e-14 
19  	1.000000e+00  	6.12178e-05 
19  	2.000000e+00  	6.96907e-05 
19  	3.000000e+00  	7.10176e-05 
19  	4.000000e+00  	7.12429e-05 
19  	5.000000e+00  	7.12822e-05 
19  	6.000000e+00  	7.12891e-05 
19  	7.000000e+00  	7.12903e-05 
19  	8.000000e+00  	7.12905e-05 
19  	9.000000e+00  	7.12906e-05 
19  	1.000000e+01  	7.12906e-05 
19  	1.100000e+01  	7.12906e-05 
19  	1.200000e+01  	7.12906e-05 
19  	1.300000e+01  	7.12906e-05 
19  	1.400000e+01  	7.12906e-05 
19  	1.500000e+01  	7.12906e-05 
19  	1.600000e+01  	7.12906e-05 
19  	1.700000e+01  	7.12906e-05 
19  	1.800000e+01  	7.12906e-05 
19  	1.900000e+01  	7.12906e-05 
19  	2.000000e+01  	7.12906e-05

20  	0.000000e+00  	-4e-14 
20  	1.000000e+00  	6.08059e-05 
20  	2.000000e+00  	6.73844e-05 
20  	3.000000e+00  	6.83733e-05 
20  	4.000000e+00  	6.85366e-05 
20  	5.000000e+00  	6.85643e-05 
20  	6.000000e+00  	6.85691e-05 
20  	7.000000e+00  	6.85700e-05 
20  	8.000000e+00  	6.85701e-05 
20  	9.000000e+00  	6.85701e-05 
20  	1.000000e+01  	6.85701e-05 
20  	1.100000e+01  	6.85701e-05 
20  	1.200000e+01  	6.85701e-05 
20  	1.300000e+01  	6.85701e-05 
20  	1.400000e+01  	6.85701e-05 
20  	1.500000e+01  	6.85701e-05 
20  	1.600000e+01  	6.85701e-05 
20  	1.700000e+01  	6.85701e-05 
20  	1.800000e+01  	6.85701e-05 
20  	1.900000e+01  	6.85701e-05 
20  	2.000000e+01  	6.85701e-05 
};
\end{axis}
\end{tikzpicture}
%
%
\begin{tikzpicture}
\begin{semilogyaxis}
[scale=0.65,
title={(b) $\frac{\sup_t |h_n(t)|}{|G_n|}$},
xlabel={n},
]
\addplot coordinates { (4 , 1.31318)  (6 , 5.26793)  (8 , 0.0137603)  (10 , 0.501182)  (12 , 0.737472)  (14 , 1.09116)  (16 , 4.13792)  (18 , 0.235436)  (20 , 0.476499)   }; 
\end{semilogyaxis}
\end{tikzpicture}
\caption{Behavior of the nonlinear part $h_n$.}
\label{Fig bi-Gauss CV}
\end{figure}
In this special case, the linear part $e^{-t\mathcal{L}}G$ and
the nonlinear part  $e^{-t\mathcal{L}}h$ have roughly the same behavior.
We present in figure \ref{Fig bi-gauss CV 1} 
the graph of the linear part and nonlinear part
and the ratio in $L^2$-norm 
\begin{equation}\label{RN=}
R_N(t) 
= \frac{\|g_N^{n\ell}(t,\cdot)\|_{L^2}}{\|g_N^{\ell in}(t,\cdot)\|_{L^2}}
= \frac
       {\left(\sum_{n=4}^N |e^{-\lambda_n t}h_n(t)|^2\right)^{\frac12}}       {\left(\sum_{n=2}^N |e^{-\lambda_n t}G_n|^2\right)^{\frac12}}.
\end{equation}
We observe that the nonlinear part is very small compared to the other. 
%
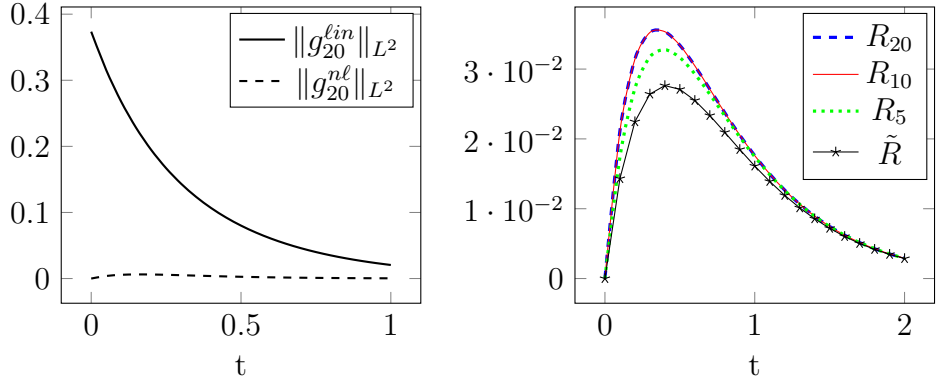
\begin{figure}[h]
\centering
\begin{tikzpicture}
\begin{axis}
[scale=0.69, 
xlabel={t},
scaled ticks=false,
legend entries={$\|g^{\ell in}_{20}\|_{L^2}$,$\|g^{n\ell}_{20}\|_{L^2}$},]
\addplot [mark=none, smooth, thick, ]  
coordinates 
{ (0 , 0.373528)  (0.05 , 0.315573)  (0.1 , 0.267719)  (0.15 , 0.227973)  (0.2 , 0.194787)  (0.25 , 0.166944)  (0.3 , 0.143481)  (0.35 , 0.123627)  (0.4 , 0.106764)  (0.45 , 0.0923892)  (0.5 , 0.0800973)  (0.55 , 0.069555)  (0.6 , 0.0604889)  (0.65 , 0.0526733)  (0.7 , 0.0459207)  (0.75 , 0.040075)  (0.8 , 0.0350053)  (0.85 , 0.0306016)  (0.9 , 0.0267708)  (0.95 , 0.0234343)  (1 , 0.020525)   }; 
\addplot [mark=none, smooth, thick, dashed]  
coordinates 
{ (0 , 1.20446e-11)  (0.05 , 0.00388195)  (0.1 , 0.00566097)  (0.15 , 0.00622724)  (0.2 , 0.00612069)  (0.25 , 0.0056666)  (0.3 , 0.00505808)  (0.35 , 0.00440686)  (0.4 , 0.00377481)  (0.45 , 0.00319355)  (0.5 , 0.00267666)  (0.55 , 0.00222731)  (0.6 , 0.00184289)  (0.65 , 0.00151786)  (0.7 , 0.0012455)  (0.75 , 0.00101886)  (0.8 , 0.000831294)  (0.85 , 0.000676755)  (0.9 , 0.000549892)  (0.95 , 0.000446062)  (1 , 0.000361301)   };
\end{axis}
\end{tikzpicture}
\,\,
\centering
\begin{tikzpicture}
\begin{axis}
[scale=0.69,
xlabel={t},
scaled ticks=false,
legend entries={$R_{20}$,$R_{10}$,$R_{5}$, $\tilde{R}$},]
\addplot[ 
mark=none, smooth, color=blue, very thick, dashed, ] 
coordinates 
{ (0 , 3.24649e-11)  (0.1 , 0.0211452)  (0.2 , 0.0314225)  (0.3 , 0.0352525)  (0.4 , 0.0353567)  (0.5 , 0.0334177)  (0.6 , 0.0304665)  (0.7 , 0.0271229)  (0.8 , 0.0237476)  (0.9 , 0.0205407)  (1 , 0.017603)  (1.1 , 0.014976)  (1.2 , 0.012666)  (1.3 , 0.0106599)  (1.4 , 0.00893415)  (1.5 , 0.00746073)  (1.6 , 0.00621047)  (1.7 , 0.00515505)  (1.8 , 0.00426803)  (1.9 , 0.00352541)  (2 , 0.0029058)   }; 
\addplot[   
mark=none, smooth, color=red,  ] 
coordinates 
{ (0 , 3.21350e-11)  (0.1 , 0.0210936)  (0.2 , 0.0313929)  (0.3 , 0.0352399)  (0.4 , 0.035352)  (0.5 , 0.033416)  (0.6 , 0.030466)  (0.7 , 0.0271227)  (0.8 , 0.0237476)  (0.9 , 0.0205407)  (1 , 0.017603)  (1.1 , 0.014976)  (1.2 , 0.012666)  (1.3 , 0.0106599)  (1.4 , 0.00893415)  (1.5 , 0.00746073)  (1.6 , 0.00621047)  (1.7 , 0.00515505)  (1.8 , 0.00426803)  (1.9 , 0.00352541)  (2 , 0.0029058)   }; 
\addplot[     
mark=none, smooth, color=green, dotted, very thick, ]
coordinates 
{ (0 , 0)  (0.1 , 0.0173991)  (0.2 , 0.0272175)  (0.3 , 0.0317416)  (0.4 , 0.0327722)  (0.5 , 0.0316474)  (0.6 , 0.0293101)  (0.7 , 0.0263924)  (0.8 , 0.0232977)  (0.9 , 0.0202688)  (1 , 0.0174411)  (1.1 , 0.0148807)  (1.2 , 0.0126105)  (1.3 , 0.0106278)  (1.4 , 0.0089157)  (1.5 , 0.00745017)  (1.6 , 0.00620445)  (1.7 , 0.00515163)  (1.8 , 0.00426609)  (1.9 , 0.00352431)  (2 , 0.00290518)   }; 
\addplot 
coordinates 
{ (0 , 0)  (0.1 , 0.0143165)  (0.2 , 0.0224569)  (0.3 , 0.0264212)  (0.4 , 0.0276331)  (0.5 , 0.027096)  (0.6 , 0.0255083)  (0.7 , 0.0233481)  (0.8 , 0.020936)  (0.9 , 0.018481)  (1 , 0.0161135)  (1.1 , 0.0139097)  (1.2 , 0.0119088)  (1.3 , 0.0101257)  (1.4 , 0.00855908)  (1.5 , 0.00719842)  (1.6 , 0.00602755)  (1.7 , 0.00502774)  (1.8 , 0.00417953)  (1.9 , 0.00346391)  (2 , 0.00286306)   }; 
\end{axis}
\end{tikzpicture}
\caption{Comparison of the non linear part with respect to the linear part 
(see \eqref{gln N=}, \eqref{gnl N=}, \eqref{RN=}, \eqref{Rn approx}).} 
\label{Fig bi-gauss CV 1}
\end{figure}
We remark that in this case the series \eqref{gln N=} and \eqref{gnl N=}
behave as 
\begin{align*}
&g^{\ell in}_N(t,v) \approx e^{-\lambda_2 \, t} \, G_2  \, \varphi_2  (v),
\\
&g^{n\ell}_N(t,v) \approx 2.51 \, e^{-\lambda_4 \, t} \, 
   G_2^2 (1- e^{-(2\,\lambda_2-\lambda_4) \, t})  \, \varphi_4 (v), 
\end{align*}  
because the terms of $h_n(t)$ are composed of products of terms which 
are numerically converging to zero. 
The quotient (for $G_2\neq0$) of the two previous approximations behaves 
closely like the ratio (see \eqref{RN=})
\begin{equation}\label{Rn approx}
 R_N(t) \approx \tilde{R}(t) \stackrel{\text{def}}{=} 
 e^{-(\lambda_4-\lambda_2) \, t}  |G_2| (1- e^{-(2\,\lambda_2-\lambda_4) \, t}).
\end{equation}
We finaly compute (see figure \ref{Fig  bi-Gauss F})
\begin{figure}[h]
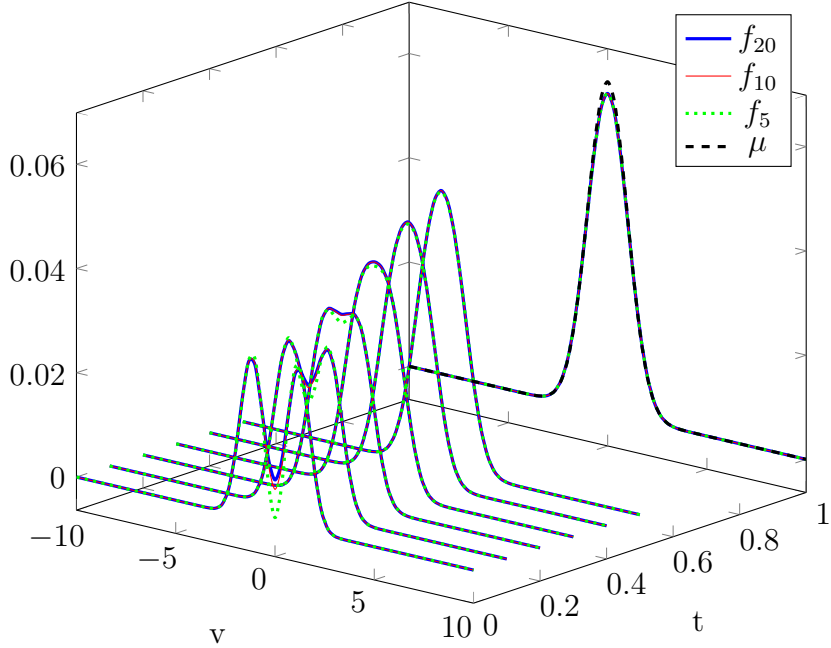

\centering

\caption{Graph of $(t,v) \mapsto f_{N}(t,v)$ for $N=5$, $20$ and $\mu(v)$.}
\label{Fig  bi-Gauss F}
\end{figure}
the solution $f=\mu+\sqrt{\mu}g$ 
using the spectral Hermite eigenfunctions $\varphi_n(v)$ 
and the expansion \eqref{f=mu+mu2 G+h} 
of $g$ in this basis. 
Since the function $g(t,\cdot) \in \mathcal{N}^{\perp}$ for all time $t\geq 0$,
the approximate solution $f_N(t,\cdot)$ is naturally orthogonal to $\varphi_0$
and $\varphi_1$. 
Therefore is a conservation of the mass and the energy.
Finally, we check that the approximate solution $f_N$ converges to the Gaussian function
when the time tends to infinity.

\section{Numerical results for initial measure data}

We consider the initial measure data 
\[\tilde{F} = {\rm gaussian} + {\rm Dirac} = \mu + \delta. \] 
Following the lemma and rescaling the solution 
$F(v) = 2^{-\frac52} \tilde{F}(2^{-\frac12}v)$, 
we get the normalized initial data   
\begin{equation}\label{G measure}
\left\{
\begin{split}
F(v) &=  2^{-\frac52} \mu(2^{-\frac12}v) + 2^{-1} \delta(v) , 
\\
G(v) &= 2^{-\frac{13}{4}} \pi^{-\frac34} - \sqrt{\mu(v)} 
  + 2^{-\frac14} \pi^{\frac34} \, \delta(v) . 
\end{split}
\right.
\end{equation}
We verify that $\langle G, \varphi_0 \rangle = \langle G, \varphi_1 \rangle = 0$
and therefore $G \in \mathcal{N}^{\perp}$. 
We then compute the spectral coefficients for $n\geq 0$ 
(see proposition \ref{Gn=}): 
\begin{align*}
G_n = \langle G, \varphi_n \rangle 
= \frac{1+(-1)^n}{2} \, 
\left( \frac{(2\,n+1)!}{2^{2n}(n!)^2}\right)^{\frac12}.
\end{align*}
Note that the coefficients $G_{2n+1}$ are equal to zero and we have 
the following approximation of $G$:  
$$G(v) \approx \sum_{n=1}^\infty  \,n^{\frac14} \varphi_{2n}(v).$$
We set $F_{\text{reg}}(v) =  2^{-\frac52} \mu(2^{-\frac12}v)$ the regular part
of the distribution $F$.
We check in the left figure \ref{Fig mu+d F2} that the approximate initial data 
behaves as a Dirac function.
%

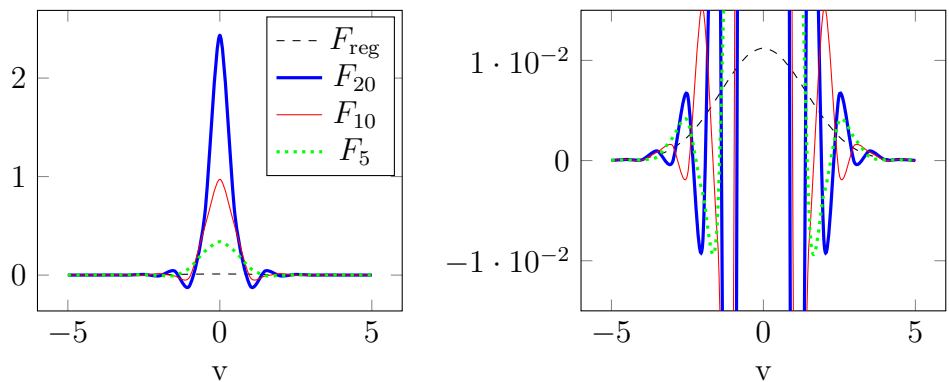
\begin{figure}[h]
\centering
\begin{tikzpicture}
\begin{axis}
[scale=.70,
xlabel={v},
scaled ticks=false,
legend entries={$F_{\text{reg}}$,$F_{20}$,$F_{10}$,$F_{5}$},
],
\addplot[ 
mark=none,
smooth,
color=black,
dashed,
] 
coordinates { (-5 , 2.16678e-05)  (-4.5 , 7.10460e-05)  (-4 , 0.000205578)  (-3.5 , 0.000524963)  (-3 , 0.00118302)  (-2.5 , 0.00235272)  (-2 , 0.00412915)  (-1.5 , 0.00639535)  (-1 , 0.00874141)  (-0.5 , 0.0105442)  (0 , 0.0112242)  (0.5 , 0.0105442)  (1 , 0.00874141)  (1.5 , 0.00639535)  (2 , 0.00412915)  (2.5 , 0.00235272)  (3 , 0.00118302)  (3.5 , 0.000524963)  (4 , 0.000205578)  (4.5 , 7.10460e-05)  (5 , 2.16678e-05)   }; 
\addplot[ 
mark=5,
smooth,
color=blue,
very thick,
] 
 coordinates { (-5 , 1.14026e-05)  (-4.5 , 0.000108477)  (-4 , 7.35889e-05)  (-3.5 , 0.000963903)  (-3 , -0.000217258)  (-2.5 , 0.00661076)  (-2 , -0.00836896)  (-1.5 , 0.0430698)  (-1 , -0.109922)  (-0.5 , 0.601972)  (0 , 2.43278)  (0.5 , 0.601972)  (1 , -0.109922)  (1.5 , 0.0430698)  (2 , -0.00836896)  (2.5 , 0.00661076)  (3 , -0.000217258)  (3.5 , 0.000963903)  (4 , 7.35889e-05)  (4.5 , 0.000108477)  (5 , 1.14026e-05)   }; 
\addplot[ 
mark=none,
smooth,
color=red,
] 
coordinates { (-5 , 2.63845e-05)  (-4.5 , 9.80485e-05)  (-4 , 4.38176e-05)  (-3.5 , 0.00083333)  (-3 , 0.00154507)  (-2.5 , -0.00145593)  (-2 , 0.0153269)  (-1.5 , -0.00828102)  (-1 , -0.0230109)  (-0.5 , 0.483845)  (0 , 0.971907)  (0.5 , 0.483845)  (1 , -0.0230109)  (1.5 , -0.00828102)  (2 , 0.0153269)  (2.5 , -0.00145593)  (3 , 0.00154507)  (3.5 , 0.00083333)  (4 , 4.38176e-05)  (4.5 , 9.80485e-05)  (5 , 2.63845e-05)   };
\addplot[ 
mark=none,
smooth,
color=green,
dotted, very thick,
]
coordinates { (-5 , 3.26115e-05)  (-4.5 , 7.09926e-05)  (-4 , 7.46046e-05)  (-3.5 , 0.000362354)  (-3 , 0.00222626)  (-2.5 , 0.00400616)  (-2 , -0.00319997)  (-1.5 , -0.00503571)  (-1 , 0.070202)  (-0.5 , 0.239317)  (0 , 0.338798)  (0.5 , 0.239317)  (1 , 0.070202)  (1.5 , -0.00503571)  (2 , -0.00319997)  (2.5 , 0.00400616)  (3 , 0.00222626)  (3.5 , 0.000362354)  (4 , 7.46046e-05)  (4.5 , 7.09926e-05)  (5 , 3.26115e-05)   }; 
\end{axis}
\end{tikzpicture}
%
\,\,
%
\centering
\begin{tikzpicture}
\begin{axis}
[scale=.70, 
ymin=-0.015, ymax=0.015,
xlabel={v},
scaled ticks=false,
],
\addplot[ 
mark=none,
smooth,
color=black,
dashed,
] 
coordinates { (-5 , 2.16678e-05)  (-4.5 , 7.10460e-05)  (-4 , 0.000205578)  (-3.5 , 0.000524963)  (-3 , 0.00118302)  (-2.5 , 0.00235272)  (-2 , 0.00412915)  (-1.5 , 0.00639535)  (-1 , 0.00874141)  (-0.5 , 0.0105442)  (0 , 0.0112242)  (0.5 , 0.0105442)  (1 , 0.00874141)  (1.5 , 0.00639535)  (2 , 0.00412915)  (2.5 , 0.00235272)  (3 , 0.00118302)  (3.5 , 0.000524963)  (4 , 0.000205578)  (4.5 , 7.10460e-05)  (5 , 2.16678e-05)   }; 
\addplot[ 
mark=none,
smooth,
color=blue,
very thick,
] 
 coordinates { (-5 , 1.14026e-05)  (-4.5 , 0.000108477)  (-4 , 7.35889e-05)  (-3.5 , 0.000963903)  (-3 , -0.000217258)  (-2.5 , 0.00661076)  (-2 , -0.00836896)  (-1.5 , 0.0430698)  (-1 , -0.109922)  (-0.5 , 0.601972)  (0 , 2.43278)  (0.5 , 0.601972)  (1 , -0.109922)  (1.5 , 0.0430698)  (2 , -0.00836896)  (2.5 , 0.00661076)  (3 , -0.000217258)  (3.5 , 0.000963903)  (4 , 7.35889e-05)  (4.5 , 0.000108477)  (5 , 1.14026e-05)   }; 
\addplot[ 
mark=none,
smooth,
color=red,
] 
coordinates { (-5 , 2.63845e-05)  (-4.5 , 9.80485e-05)  (-4 , 4.38176e-05)  (-3.5 , 0.00083333)  (-3 , 0.00154507)  (-2.5 , -0.00145593)  (-2 , 0.0153269)  (-1.5 , -0.00828102)  (-1 , -0.0230109)  (-0.5 , 0.483845)  (0 , 0.971907)  (0.5 , 0.483845)  (1 , -0.0230109)  (1.5 , -0.00828102)  (2 , 0.0153269)  (2.5 , -0.00145593)  (3 , 0.00154507)  (3.5 , 0.00083333)  (4 , 4.38176e-05)  (4.5 , 9.80485e-05)  (5 , 2.63845e-05)   };
\addplot[ 
mark=none,
smooth,
color=green,
dotted, very thick,
]
coordinates { (-5 , 3.26115e-05)  (-4.5 , 7.09926e-05)  (-4 , 7.46046e-05)  (-3.5 , 0.000362354)  (-3 , 0.00222626)  (-2.5 , 0.00400616)  (-2 , -0.00319997)  (-1.5 , -0.00503571)  (-1 , 0.070202)  (-0.5 , 0.239317)  (0 , 0.338798)  (0.5 , 0.239317)  (1 , 0.070202)  (1.5 , -0.00503571)  (2 , -0.00319997)  (2.5 , 0.00400616)  (3 , 0.00222626)  (3.5 , 0.000362354)  (4 , 7.46046e-05)  (4.5 , 7.09926e-05)  (5 , 3.26115e-05)   }; 
\end{axis}
\end{tikzpicture}
\caption{Approximation of the initial data.}
\label{Fig mu+d F2}
\end{figure}
Remark that to capture the approximation of the regular part $F_{\text{reg}}$, 
we have to rescale the cote $y$-coordinate. 
We observe the oscillations of $F_N$ which are 
expected since the functions $F_N$ approach the Dirac function when $N$ 
tends to infinity (see the right figure \ref{Fig mu+d F2}).
We now focus on the evolution problem. 
As the initial data is a distribution, we can check that the 
linear part of the solution is singular~:  
\begin{equation}\label{glin sing}
\|g^{\ell in}(t,\cdot)\|_{L^2}^2  
= \sum_{n=2}^\infty   G_n^2 e^{-2\,\lambda_n t }
\approx  \frac{1}{t^{\alpha}} , 
\quad \text{when}\quad t\to 0
\end{equation}
for some $\alpha>0$ 
(since $G_{2n} \approx n^{\frac14}$ and $\lambda_n \approx n^{\frac12}$). 
We next compute the nonlinear part $h_n(t)$ of the solution 
(see the left figure \ref{Fig mu+d CV}).
\begin{figure}[h]
\centering
\begin{tikzpicture}
\begin{axis}
[scale=0.65,
view={-20}{30},
title= {$t \mapsto h_n(t)$ for $n=4,5,\ldots,20$.},
xlabel={$n$},
ylabel={$t$},
scaled ticks=false,
xtick={4,8,12,16,20},
zticklabels={$0$,$.05$,$0.1$,$.15$}, 
]
\addplot3 table [mark=none, ]{
4  	0.000000e+00  	0 
4  	5.000000e-01  	0.614179 
4  	1.000000e+00  	1.14821 
4  	1.500000e+00  	1.61254 
4  	2.000000e+00  	2.01628 
4  	2.500000e+00  	2.36732 
4  	3.000000e+00  	2.67256 
4  	3.500000e+00  	2.93796 
4  	4.000000e+00  	3.16872 
4  	4.500000e+00  	3.36937 
4  	5.000000e+00  	3.54383 
4  	5.500000e+00  	3.69553 
4  	6.000000e+00  	3.82743 
4  	6.500000e+00  	3.94211 
4  	7.000000e+00  	4.04183 
4  	7.500000e+00  	4.12853 
4  	8.000000e+00  	4.20392 
4  	8.500000e+00  	4.26947 
4  	9.000000e+00  	4.32647 
4  	9.500000e+00  	4.37603 
4  	1.000000e+01  	4.41912

5  	0.000000e+00  	0 
5  	5.000000e-01  	0 
5  	1.000000e+00  	0 
5  	1.500000e+00  	0 
5  	2.000000e+00  	0 
5  	2.500000e+00  	0 
5  	3.000000e+00  	0 
5  	3.500000e+00  	0 
5  	4.000000e+00  	0 
5  	4.500000e+00  	0 
5  	5.000000e+00  	0 
5  	5.500000e+00  	0 
5  	6.000000e+00  	0 
5  	6.500000e+00  	0 
5  	7.000000e+00  	0 
5  	7.500000e+00  	0 
5  	8.000000e+00  	0 
5  	8.500000e+00  	0 
5  	9.000000e+00  	0 
5  	9.500000e+00  	0 
5  	1.000000e+01  	0

6  	0.000000e+00  	-3.28634e-09 
6  	5.000000e-01  	1.10211 
6  	1.000000e+00  	2.00915 
6  	1.500000e+00  	2.69625 
6  	2.000000e+00  	3.19198 
6  	2.500000e+00  	3.53841 
6  	3.000000e+00  	3.77518 
6  	3.500000e+00  	3.93437 
6  	4.000000e+00  	4.04009 
6  	4.500000e+00  	4.10963 
6  	5.000000e+00  	4.15502 
6  	5.500000e+00  	4.18447 
6  	6.000000e+00  	4.20347 
6  	6.500000e+00  	4.21569 
6  	7.000000e+00  	4.22352 
6  	7.500000e+00  	4.22852 
6  	8.000000e+00  	4.23171 
6  	8.500000e+00  	4.23374 
6  	9.000000e+00  	4.23503 
6  	9.500000e+00  	4.23584 
6  	1.000000e+01  	4.23636

7  	0.000000e+00  	0 
7  	5.000000e-01  	0 
7  	1.000000e+00  	0 
7  	1.500000e+00  	0 
7  	2.000000e+00  	0 
7  	2.500000e+00  	0 
7  	3.000000e+00  	0 
7  	3.500000e+00  	0 
7  	4.000000e+00  	0 
7  	4.500000e+00  	0 
7  	5.000000e+00  	0 
7  	5.500000e+00  	0 
7  	6.000000e+00  	0 
7  	6.500000e+00  	0 
7  	7.000000e+00  	0 
7  	7.500000e+00  	0 
7  	8.000000e+00  	0 
7  	8.500000e+00  	0 
7  	9.000000e+00  	0 
7  	9.500000e+00  	0 
7  	1.000000e+01  	0

8  	0.000000e+00  	3.46334e-09 
8  	5.000000e-01  	1.55959 
8  	1.000000e+00  	2.77329 
8  	1.500000e+00  	3.58313 
8  	2.000000e+00  	4.08058 
8  	2.500000e+00  	4.3718 
8  	3.000000e+00  	4.53746 
8  	3.500000e+00  	4.63007 
8  	4.000000e+00  	4.68131 
8  	4.500000e+00  	4.70949 
8  	5.000000e+00  	4.72494 
8  	5.500000e+00  	4.73339 
8  	6.000000e+00  	4.73802 
8  	6.500000e+00  	4.74055 
8  	7.000000e+00  	4.74193 
8  	7.500000e+00  	4.74269 
8  	8.000000e+00  	4.7431 
8  	8.500000e+00  	4.74333 
8  	9.000000e+00  	4.74345 
8  	9.500000e+00  	4.74352 
8  	1.000000e+01  	4.74356

9  	0.000000e+00  	0 
9  	5.000000e-01  	0 
9  	1.000000e+00  	0 
9  	1.500000e+00  	0 
9  	2.000000e+00  	0 
9  	2.500000e+00  	0 
9  	3.000000e+00  	0 
9  	3.500000e+00  	0 
9  	4.000000e+00  	0 
9  	4.500000e+00  	0 
9  	5.000000e+00  	0 
9  	5.500000e+00  	0 
9  	6.000000e+00  	0 
9  	6.500000e+00  	0 
9  	7.000000e+00  	0 
9  	7.500000e+00  	0 
9  	8.000000e+00  	0 
9  	8.500000e+00  	0 
9  	9.000000e+00  	0 
9  	9.500000e+00  	0 
9  	1.000000e+01  	0

10  	0.000000e+00  	-4.38178e-09 
10  	5.000000e-01  	2.00343 
10  	1.000000e+00  	3.49025 
10  	1.500000e+00  	4.38617 
10  	2.000000e+00  	4.87485 
10  	2.500000e+00  	5.12877 
10  	3.000000e+00  	5.25775 
10  	3.500000e+00  	5.32263 
10  	4.000000e+00  	5.35517 
10  	4.500000e+00  	5.37149 
10  	5.000000e+00  	5.37967 
10  	5.500000e+00  	5.38378 
10  	6.000000e+00  	5.38584 
10  	6.500000e+00  	5.38688 
10  	7.000000e+00  	5.38741 
10  	7.500000e+00  	5.38767 
10  	8.000000e+00  	5.3878 
10  	8.500000e+00  	5.38787 
10  	9.000000e+00  	5.3879 
10  	9.500000e+00  	5.38792 
10  	1.000000e+01  	5.38793

11  	0.000000e+00  	0 
11  	5.000000e-01  	0 
11  	1.000000e+00  	0 
11  	1.500000e+00  	0 
11  	2.000000e+00  	0 
11  	2.500000e+00  	0 
11  	3.000000e+00  	0 
11  	3.500000e+00  	0 
11  	4.000000e+00  	0 
11  	4.500000e+00  	0 
11  	5.000000e+00  	0 
11  	5.500000e+00  	0 
11  	6.000000e+00  	0 
11  	6.500000e+00  	0 
11  	7.000000e+00  	0 
11  	7.500000e+00  	0 
11  	8.000000e+00  	0 
11  	8.500000e+00  	0 
11  	9.000000e+00  	0 
11  	9.500000e+00  	0 
11  	1.000000e+01  	0

12  	0.000000e+00  	4.90453e-09 
12  	5.000000e-01  	2.44005 
12  	1.000000e+00  	4.18237 
12  	1.500000e+00  	5.15244 
12  	2.000000e+00  	5.63882 
12  	2.500000e+00  	5.87265 
12  	3.000000e+00  	5.98343 
12  	3.500000e+00  	6.03572 
12  	4.000000e+00  	6.06042 
12  	4.500000e+00  	6.0721 
12  	5.000000e+00  	6.07763 
12  	5.500000e+00  	6.08025 
12  	6.000000e+00  	6.08149 
12  	6.500000e+00  	6.08208 
12  	7.000000e+00  	6.08237 
12  	7.500000e+00  	6.0825 
12  	8.000000e+00  	6.08256 
12  	8.500000e+00  	6.08259 
12  	9.000000e+00  	6.08261 
12  	9.500000e+00  	6.08261 
12  	1.000000e+01  	6.08262

13  	0.000000e+00  	0 
13  	5.000000e-01  	0 
13  	1.000000e+00  	0 
13  	1.500000e+00  	0 
13  	2.000000e+00  	0 
13  	2.500000e+00  	0 
13  	3.000000e+00  	0 
13  	3.500000e+00  	0 
13  	4.000000e+00  	0 
13  	4.500000e+00  	0 
13  	5.000000e+00  	0 
13  	5.500000e+00  	0 
13  	6.000000e+00  	0 
13  	6.500000e+00  	0 
13  	7.000000e+00  	0 
13  	7.500000e+00  	0 
13  	8.000000e+00  	0 
13  	8.500000e+00  	0 
13  	9.000000e+00  	0 
13  	9.500000e+00  	0 
13  	1.000000e+01  	0

14  	0.000000e+00  	2.90073e-09 
14  	5.000000e-01  	2.87302 
14  	1.000000e+00  	4.8623 
14  	1.500000e+00  	5.90453 
14  	2.000000e+00  	6.39673 
14  	2.500000e+00  	6.62123 
14  	3.000000e+00  	6.72271 
14  	3.500000e+00  	6.76855 
14  	4.000000e+00  	6.78929 
14  	4.500000e+00  	6.79868 
14  	5.000000e+00  	6.80294 
14  	5.500000e+00  	6.80487 
14  	6.000000e+00  	6.80575 
14  	6.500000e+00  	6.80615 
14  	7.000000e+00  	6.80633 
14  	7.500000e+00  	6.80641 
14  	8.000000e+00  	6.80645 
14  	8.500000e+00  	6.80647 
14  	9.000000e+00  	6.80647 
14  	9.500000e+00  	6.80648 
14  	1.000000e+01  	6.80648

15  	0.000000e+00  	0 
15  	5.000000e-01  	0 
15  	1.000000e+00  	0 
15  	1.500000e+00  	0 
15  	2.000000e+00  	0 
15  	2.500000e+00  	0 
15  	3.000000e+00  	0 
15  	3.500000e+00  	0 
15  	4.000000e+00  	0 
15  	4.500000e+00  	0 
15  	5.000000e+00  	0 
15  	5.500000e+00  	0 
15  	6.000000e+00  	0 
15  	6.500000e+00  	0 
15  	7.000000e+00  	0 
15  	7.500000e+00  	0 
15  	8.000000e+00  	0 
15  	8.500000e+00  	0 
15  	9.000000e+00  	0 
15  	9.500000e+00  	0 
15  	1.000000e+01  	0

16  	0.000000e+00  	6.88468e-09 
16  	5.000000e-01  	3.30468 
16  	1.000000e+00  	5.53794 
16  	1.500000e+00  	6.65442 
16  	2.000000e+00  	7.1592 
16  	2.500000e+00  	7.3809 
16  	3.000000e+00  	7.47768 
16  	3.500000e+00  	7.51994 
16  	4.000000e+00  	7.53842 
16  	4.500000e+00  	7.54651 
16  	5.000000e+00  	7.55006 
16  	5.500000e+00  	7.55162 
16  	6.000000e+00  	7.5523 
16  	6.500000e+00  	7.5526 
16  	7.000000e+00  	7.55273 
16  	7.500000e+00  	7.55279 
16  	8.000000e+00  	7.55281 
16  	8.500000e+00  	7.55282 
16  	9.000000e+00  	7.55283 
16  	9.500000e+00  	7.55283 
16  	1.000000e+01  	7.55283

17  	0.000000e+00  	0 
17  	5.000000e-01  	0 
17  	1.000000e+00  	0 
17  	1.500000e+00  	0 
17  	2.000000e+00  	0 
17  	2.500000e+00  	0 
17  	3.000000e+00  	0 
17  	3.500000e+00  	0 
17  	4.000000e+00  	0 
17  	4.500000e+00  	0 
17  	5.000000e+00  	0 
17  	5.500000e+00  	0 
17  	6.000000e+00  	0 
17  	6.500000e+00  	0 
17  	7.000000e+00  	0 
17  	7.500000e+00  	0 
17  	8.000000e+00  	0 
17  	8.500000e+00  	0 
17  	9.000000e+00  	0 
17  	9.500000e+00  	0 
17  	1.000000e+01  	0

18  	0.000000e+00  	5.58496e-09 
18  	5.000000e-01  	3.7367 
18  	1.000000e+00  	6.21453 
18  	1.500000e+00  	7.409 
18  	2.000000e+00  	7.93163 
18  	2.500000e+00  	8.15461 
18  	3.000000e+00  	8.24931 
18  	3.500000e+00  	8.28954 
18  	4.000000e+00  	8.30665 
18  	4.500000e+00  	8.31394 
18  	5.000000e+00  	8.31705 
18  	5.500000e+00  	8.31837 
18  	6.000000e+00  	8.31893 
18  	6.500000e+00  	8.31917 
18  	7.000000e+00  	8.31928 
18  	7.500000e+00  	8.31932 
18  	8.000000e+00  	8.31934 
18  	8.500000e+00  	8.31935 
18  	9.000000e+00  	8.31935 
18  	9.500000e+00  	8.31935 
18  	1.000000e+01  	8.31935

19  	0.000000e+00  	0 
19  	5.000000e-01  	0 
19  	1.000000e+00  	0 
19  	1.500000e+00  	0 
19  	2.000000e+00  	0 
19  	2.500000e+00  	0 
19  	3.000000e+00  	0 
19  	3.500000e+00  	0 
19  	4.000000e+00  	0 
19  	4.500000e+00  	0 
19  	5.000000e+00  	0 
19  	5.500000e+00  	0 
19  	6.000000e+00  	0 
19  	6.500000e+00  	0 
19  	7.000000e+00  	0 
19  	7.500000e+00  	0 
19  	8.000000e+00  	0 
19  	8.500000e+00  	0 
19  	9.000000e+00  	0 
19  	9.500000e+00  	0 
19  	1.000000e+01  	0

20  	0.000000e+00  	2.76690e-08 
20  	5.000000e-01  	4.17033 
20  	1.000000e+00  	6.8957 
20  	1.500000e+00  	8.1725 
20  	2.000000e+00  	8.71709 
20  	2.500000e+00  	8.94411 
20  	3.000000e+00  	9.03834 
20  	3.500000e+00  	9.07747 
20  	4.000000e+00  	9.09373 
20  	4.500000e+00  	9.10049 
20  	5.000000e+00  	9.1033 
20  	5.500000e+00  	9.10447 
20  	6.000000e+00  	9.10496 
20  	6.500000e+00  	9.10517 
20  	7.000000e+00  	9.10525 
20  	7.500000e+00  	9.10529 
20  	8.000000e+00  	9.1053 
20  	8.500000e+00  	9.10531 
20  	9.000000e+00  	9.10531 
20  	9.500000e+00  	9.10531 
20  	1.000000e+01  	9.10531 
};
\end{axis}
\end{tikzpicture}
\quad
\begin{tikzpicture}
\begin{axis}
[scale=0.65,
title={$\frac{\sup_t |h_n(t)|}{|G_n|}$ for $n=4,6,\ldots,20$.},
xlabel={n},
]
\addplot coordinates { (4 , 3)  (6 , 2.47432)  (8 , 2.59618)  (10 , 2.801)  (12 , 3.03015)  (14 , 3.2695)  (16 , 3.51452)  (18 , 3.76359)  (20 , 4.0161)   }; 
\end{axis}
\end{tikzpicture}
\caption{Behavior of the nonlinear part $h_n$.}
\label{Fig mu+d CV}
\end{figure}
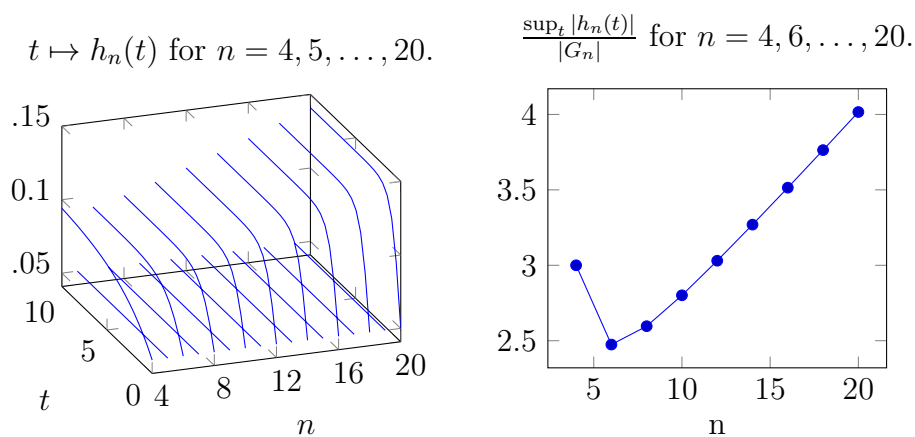

We observe some numerical evidences that these functions 
are increasing less than a power of $n$ : 
\[\sup_{t\geq0} |h_n(t)| \leq C \, n^{a}\]   
with $a$ close to 1. 
Since $G_n \approx n^{\frac14}$, 
the behavior of a term of the series $(g^{\ell in}_N(t) + g^{n\ell}_N(t))$ 
is dominated by the nonlinear part. 

We next calculate the linear part and nonlinear part of solution 
(see Figure \ref{Fig mu+d CV 1}). 
%
%
\begin{figure}[h]
\centering
\begin{tikzpicture}
\begin{axis}
[scale=0.65, 
xlabel={t},
scaled ticks=false,
legend entries={$\|g^{\ell in}_{20}\|_{L^2}$,$\|g^{n\ell}_{20}\|_{L^2}$},]
\addplot [mark=none, smooth, thick, ]  
coordinates 
{ (0.25 , 1.00943)  (0.3 , 0.812572)  (0.35 , 0.670672)  (0.4 , 0.563324)  (0.45 , 0.478992)  (0.5 , 0.410835)  (0.55 , 0.35458)  (0.6 , 0.307422)  (0.65 , 0.267433)  (0.7 , 0.233229)  (0.75 , 0.203785)  (0.8 , 0.178315)  (0.85 , 0.1562)  (0.9 , 0.136943)  (0.95 , 0.12014)  (1 , 0.105452)  }; 
\addplot [mark=none, smooth, thick, dashed]  
coordinates 
{ (0 , 2.31255e-08)  (0.05 , 0.450922)  (0.1 , 0.519454)  (0.15 , 0.458026)  (0.2 , 0.368128)  (0.25 , 0.285534)  (0.3 , 0.219308)  (0.35 , 0.168889)  (0.4 , 0.131107)  (0.45 , 0.102734)  (0.5 , 0.0812048)  (0.55 , 0.0646533)  (0.6 , 0.0517662)  (0.65 , 0.041622)  (0.7 , 0.0335663)  (0.75 , 0.0271255)  (0.8 , 0.0219498)  (0.85 , 0.0177754)  (0.9 , 0.0143998)  (0.95 , 0.0116656)  (1 , 0.00944851)    }; 
\end{axis}
\end{tikzpicture}
\,\,
\centering
\begin{tikzpicture}
\begin{axis}
[scale=0.65,
xlabel={t},
scaled ticks=false,
legend entries={$R_{20}$,$R_{10}$,$R_{5}$,$\tilde{R}$},]
\addplot[ 
mark=none,
smooth,
color=blue, very thick,
] 
coordinates 
{ (0 , 3.30467e-09)  (0.08 , 0.176298)  (0.16 , 0.268237)  (0.24 , 0.284267)  (0.32 , 0.262983)  (0.4 , 0.232738)  (0.48 , 0.204203)  (0.56 , 0.179444)  (0.64 , 0.158098)  (0.72 , 0.139493)  (0.8 , 0.123096)  (0.88 , 0.108536)  (0.96 , 0.0955575)  (1.04 , 0.0839716)  (1.12 , 0.0736333)  (1.2 , 0.0644228)  (1.28 , 0.0562356)  (1.36 , 0.0489778)  (1.44 , 0.0425624)  (1.52 , 0.0369087)  (1.6 , 0.0319409)  (1.68 , 0.0275884)  (1.76 , 0.0237854)  (1.84 , 0.0204714)  (1.92 , 0.0175906)  (2 , 0.0150923)   }; 
\addplot[   
mark=none,
smooth,
color=red,
] 
coordinates 
{ (0 , 1.26335e-09)  (0.08 , 0.109089)  (0.16 , 0.18159)  (0.24 , 0.21679)  (0.32 , 0.223399)  (0.4 , 0.213218)  (0.48 , 0.195565)  (0.56 , 0.175876)  (0.64 , 0.15669)  (0.72 , 0.138954)  (0.8 , 0.122894)  (0.88 , 0.108462)  (0.96 , 0.0955302)  (1.04 , 0.0839617)  (1.12 , 0.0736298)  (1.2 , 0.0644215)  (1.28 , 0.0562352)  (1.36 , 0.0489776)  (1.44 , 0.0425624)  (1.52 , 0.0369087)  (1.6 , 0.0319409)  (1.68 , 0.0275884)  (1.76 , 0.0237854)  (1.84 , 0.0204714)  (1.92 , 0.0175906)  (2 , 0.0150923)   }; 
\addplot[     
mark=none,
smooth,
color=green, dotted, very thick,
]
coordinates { (0 , 0)  (0.08 , 0.0458135)  (0.16 , 0.0816415)  (0.24 , 0.107404)  (0.32 , 0.123833)  (0.4 , 0.132264)  (0.48 , 0.134321)  (0.56 , 0.131639)  (0.64 , 0.125674)  (0.72 , 0.117622)  (0.8 , 0.108408)  (0.88 , 0.0987095)  (0.96 , 0.0890054)  (1.04 , 0.0796169)  (1.12 , 0.0707477)  (1.2 , 0.062516)  (1.28 , 0.0549789)  (1.36 , 0.0481516)  (1.44 , 0.0420206)  (1.52 , 0.0365541)  (1.6 , 0.0317094)  (1.68 , 0.0274375)  (1.76 , 0.0236873)  (1.84 , 0.0204077)  (1.92 , 0.0175493)  (2 , 0.0150655) }; 
\addplot coordinates 
{ (0 , 0)  (0.1 , 0.0752025)  (0.2 , 0.117963)  (0.3 , 0.138786)  (0.4 , 0.145152)  (0.5 , 0.142331)  (0.6 , 0.133991)  (0.7 , 0.122644)  (0.8 , 0.109974)  (0.9 , 0.0970777)  (1 , 0.0846417)  (1.1 , 0.0730655)  (1.2 , 0.0625553)  (1.3 , 0.0531884)  (1.4 , 0.0449595)  (1.5 , 0.0378122)  (1.6 , 0.0316618)  (1.7 , 0.0264099)  (1.8 , 0.0219544)  (1.9 , 0.0181954)  (2 , 0.0150392)   };
\end{axis}
\end{tikzpicture}
\caption{Comparaison of the non linear part with respect to the linear part 
(see \eqref{gln N=}, \eqref{gnl N=}, \eqref{RN=}, \eqref{Rn approx}).}
\label{Fig mu+d CV 1}
\end{figure}
The numerical computations in the left figure 
show that $g^{n\ell}_N(t)$ is a regular function for all time
and we verify also that $g^{\ell in}_N(t)$ 
is singular as $t\to0$ as pointed in \eqref{glin sing}.
For a large time, the $L^2$ norm of the linear part ($\sim e^{-\lambda_2\,t}$)
dominates the norm of the nonlinear part ($\sim e^{-\lambda_4\,t}$), 
and the ration $R_N(t)$ has the same behavior as in the previous section
(see \eqref{Rn approx}). 

We then compute the numerical approximation $f_N$ of the solution $f$ for $N=20$ 
and we check that the solution behaves as a Dirac function as $t\to0$
and tends to the Gaussian as $t\to\infty$ (see Figure  \ref{Fig  mu+d F}).
\begin{figure}[h]
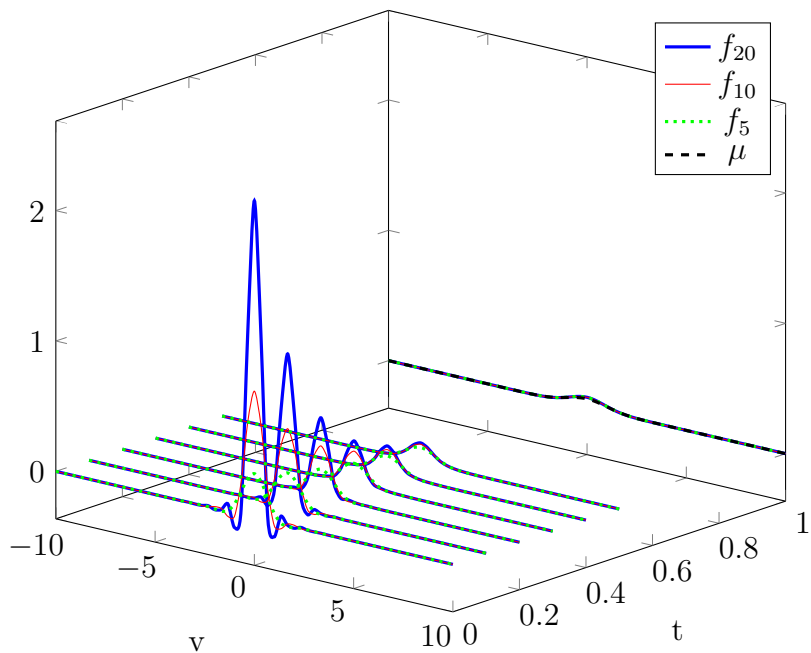

\centering

\caption{Graph of $(t,v) \mapsto f_{N}(t,v)$ for $N=5$, $10$, $20$ and $\mu(v)$.}
\label{Fig  mu+d F}
\end{figure}

Since $\lambda_n \approx c\,\sqrt{n}$ and
if the behavior of $\sup_{t\geq0} |h_n(t)|$ is dominated by a power of $n$
(which is numerically verified), 
then we have for some $b, \gamma>0$ : 
\[\forall t>0,\,\,
 \|f_N(t,\cdot) - f(t,\cdot) \|_{L^2} 
\lesssim \frac{1}{t^b}  \, e^{ - \gamma \sqrt{N} \, t}
\to 0 \quad\text{as}\quad N\to\infty. \]
We observe some other numerical evidences that the series 
converges in $L^2$ for $t>0$
and the solution converges to a Gaussian as $t\to\infty$.


\section{Conclusion}
We have considered the perturbation $g$ of the solution $f$
of the Boltzmann equation defined by 
$$f = \mu + \sqrt{\mu} g \quad\text{where}\quad
g(t,v) = {\sum_n} g_n(t) \varphi_n(v)$$ 
and we have studied the behavior of the spectral coefficients 
$$g_n(t)=  e^{-\lambda_n t}(G_n+h_n(t)), \quad g_n(0) = G_n. $$
We have then computed formally the spectral coefficients $h_n(t)$ 
for $n=0,1,\ldots,N$ with $N=20$. 
We have checked also the results for small $L^2$ initial data 
and distribution type initial data $\mu + \delta$.

$\bullet$ For small $L^2$ initial data,
our method was tested with several $L^2$ initial conditions : 
F is a sum of two Gaussian, $G_n = \frac{0.1}{n}$, $G_n = \frac{1}{n}$.
The results show that there are some numerical evidences that the spectral series 
$\sum_n e^{-\lambda_n t}(G_n+h_n(t)) \varphi_n$ is convergent in $L^2$ for any time $t\geq 0$ 
and the solution converges to a Gaussian. 
Moreover, for large times, the linear part $G_n$ is preponderant with respect 
to the non-linear part $h_n(t)$.

$\bullet$ For the distribution type initial data $\mu + \delta$,
the simulations show some numerical evidences that the spectral series converges in $L^2$ 
and there is a regularization of the solution for $t> 0$.

\medskip

We have computed the formal solutions of the spectral coefficients $h_n$ of the solution 
of the Boltzmann equation. If there exists a regular solution for $t> 0$, 
then the solutions $h_n$ are the exact projections of the solution on the spectral basis.
These calculations were made in the case of a non-cutoff kernel.
The numerical results are coherent for small $L^2$ initial data 
or for the distribution case $\mu + \delta$.  
There is conservation of the mass, momentum and energy of the approximated solution
(since $g_N(t,\cdot)$ is orthogonal to the kernel $\mathcal{N}$ for all time).
Moreover the approximated solution $f_N(t,\cdot)$ (defined in \eqref{fN=}) 
converges to a Gaussian when $t$ tends to infinity.

\section*{Acknowledgments}
The authors wish to thank Chao-Jiang Xu for interesting discussions.


\section{Appendix}

\subsection{Rescaling of the solution}\label{appendix rescaling}
We consider a radial solution $\tilde{f}(s,w)$ of the Boltzmann equation :
\begin{equation*}
\left\{
\begin{array}{ll}
   \partial_s \tilde{f} = {\bf {Q}}(\tilde{f} ,\tilde{f} ),\\
 \tilde{f} |_{t=0}=\tilde{F} .
\end{array}
\right.
\end{equation*}

\begin{lemma}\label{lem G ortho N}
We consider the functions $f(t,v)$ and $F(v)$ defined by the change of variable 
\begin{equation*}
\left\{
\begin{array}{ll}
   f(t,v) =  \alpha \tilde{f}(\frac{\alpha}{\beta^3} t, \beta v),\\
    F(v) = \alpha \tilde{F}(\beta v) 
\end{array}
\right.
\end{equation*} 
where
\begin{equation*}
\alpha = 
\frac{\left( \frac13 \int_{\mathbb{R}^3}w^2\,\tilde{F}(w)\,dw \right)^{\frac32} }
      {  \left(\int_{\mathbb{R}^3}\tilde{F}(w)\,dw\right)^{\frac52}} 
\quad\text{and}\quad
\beta = 
\frac{\left(\frac13  \int_{\mathbb{R}^3}w^2\,\tilde{F}(w)\,dw \right)^{\frac12} }
      { \left(\int_{\mathbb{R}^3}\tilde{F}(w)\,dw\right)^{\frac12}}.
\end{equation*} 
Therefore $f(t,v)$ is a solution of the Boltzmann equation \eqref{eq Boltz f} 
with initial data $F$. Moreover, if we set $F = \mu + \sqrt{\mu} \, G$,
we then have $G\in \mathcal{N}^\perp$.
\end{lemma}
\begin{remark}
If $F$ is such that 
\begin{align*}
&\int_{\mathbb{R}^3}F(v)\,dv
=\int_{\mathbb{R}^3}\mu(v)\,dv = 1,
\\
&\int_{\mathbb{R}^3}v^2\,F(v)\,dv
=\int_{\mathbb{R}^3}v^2\,\mu(v)\,dv = 3 , 
\end{align*}
then the function $G$ defined by $G = \frac{1}{\sqrt{\mu}}(F-\mu)$ 
belongs to $\mathcal{N}^\perp$.
\end{remark}

\begin{proof}
It is easy to check that $f(t,v)$ is a solution of the Boltzmann equation.
Since $G$ is a radial function, it is enough to check that
\[ \Big(G , \sqrt{\mu}\Big)_{L^2}  = \Big(G , |v|^2\sqrt{\mu}\Big)_{L^2} = 0. \]
Recalling that  $\Big(\varphi_p ,  \varphi_q \Big)_{L^2} = \delta_{pq}$ ,
\[ \varphi_0 = \sqrt{\mu} \quad\text{and}\quad \varphi_1 = 6^{-\frac12} (3-|v|^2) \sqrt{\mu}, \]
it is equivalent to prove
\begin{equation*}
  (F/\sqrt{\mu},\varphi_0)_{L^2}=1 \quad\text{and}\quad (F/\sqrt{\mu},\varphi_1)_{L^2}=0,
\end{equation*}
which gives the equations
\begin{equation*}
\left\{
\begin{array}{ll}
  \int_{\mathbb{R}^3} F(v) \, dv = \int_{\mathbb{R}^3} \mu(v) \, dv = 1 ,\\
  \int_{\mathbb{R}^3} |v|^2 \,  F(v) \, dv = \int_{\mathbb{R}^3} |v|^2 \, \mu(v) \, dv = 3.
\end{array}
\right.
\end{equation*} 
Using the change of variable $w = \beta \, v$, 
we can check that if we set the values of $\alpha$ and $\beta$ given in the lemma,
the previous equations are fulfilled.
\end{proof}

\subsection{Measure initial data}
We define the following distribution initial data : 
\begin{equation*}
\tilde{F} = \mu + \delta . 
\end{equation*}
Following the rescaling of lemma \ref{lem G ortho N}, 
we compute
\begin{align*}
\langle \tilde{F} , 1 \rangle 
  &= \int_{\mathbb{R}^3} \mu(v) \,1\,dv + \langle \delta , 1 \rangle = 2 , 
\\
\langle \tilde{F} , v^2 \rangle 
  &= \frac13 \int_{\mathbb{R}^3} \mu(v) \, v^2 \,dv + \langle \delta , v^2 \rangle = 1  
\end{align*}
and then $\alpha=2^{-\frac52}$ and $\beta=2^{-\frac12}$.
Using the change of variable $w=\beta v$, 
we get the new rescaled distribution initial data 
\begin{equation*}
F = \alpha \, \tilde{F} \circ (\beta \, {\rm Id}) 
= 2^{-\frac52} \left( \mu(2^{-\frac12} \cdot ) + (2^{\frac12})^3 \delta \right).
\end{equation*}
%

\begin{proposition}\label{prop F=mu+d}
We consider the initial data 
\begin{equation*}
F = 2^{-\frac52} \left( \mu(2^{-\frac12} \cdot ) + (2^{\frac12})^3 \delta \right) 
\end{equation*}
and we set $G$ such that $F = \mu + \sqrt{\mu} G$.
Then we have
\begin{align*}
G &= - \sqrt{\mu} + 2^{-\frac{13}{4}} \pi^{-\frac34} 
+ 2^{-\frac{1}{4}} \pi^{\frac34} \delta.
\end{align*} 
We consider the coordinates $G_n = \langle G , \varphi_n \rangle$
of the distribution $G$ in the spectral basis $(\varphi_n)_n$.
We can check that 
\[ G_0   = G_1   = 0\]
and for all integer $n\geq 2$,
\begin{equation}\label{Gn=}
G_{n} =  \langle G , \varphi_n \rangle
= \frac{1+(-1)^n}{2} \, 
\left( \frac{(2\,n+1)!}{2^{2n}(n!)^2}\right)^{\frac12}.
\end{equation}
\end{proposition}
\begin{proof}
The expression of $G$ follows from the definition of the Gaussian $\mu$.
We then compute
\[
\langle G , \varphi_n \rangle
= - ( \varphi_0 , \varphi_n )_{L^2}
+
2^{-\frac{13}{4}} \pi^{-\frac34} 
( 1 , \varphi_n )_{L^2}
+ 2^{-\frac{1}{4}} \pi^{\frac34} \varphi_n(0) 
\]
and the conclusion results directly from lemma \ref{calcul phi_n}. 
\end{proof}
We consider now a special Gaussian approximation $F_\varepsilon \in L^2$ of the 
distribution initial data $F=\mu+\delta$
and we obtain some spectral stability result in this case. 

\begin{proposition}
We consider the initial data 
for $\varepsilon>0$
\begin{equation*}
\tilde{F}_\varepsilon (w) = \mu(w) + \frac{1}{\varepsilon^3} \, 
  \mu\left( \frac{w}{\varepsilon}\right) . 
\end{equation*}
Following lemma \ref{lem G ortho N}, 
the rescaled initial data of $\tilde{F}_\varepsilon$ is 
$F_\varepsilon = \mu + \sqrt{\mu} \, G_\varepsilon$
where $G_\varepsilon \in \mathcal{N}^\perp$ and 
\begin{align*}
G_\varepsilon(v) =
- \sqrt{\mu(v)} 
+ 2^{-\frac52} \left(1+\varepsilon^{2}\right)^{3/2} 
\left(
\sqrt{\mu(\varepsilon v)} 
+ 
\frac{1}{\varepsilon^3}
\sqrt{\mu(v/\varepsilon)} 
 \right).
\end{align*} 
Then we have the following limit in the sense of distribution
as $\varepsilon\to0$: 
\begin{align*}
F_\varepsilon  &\to F = 
 2^{-\frac52} \left( \mu(2^{-\frac12} \cdot ) + (2^{\frac12})^3 \delta \right) , 
\\
G_\varepsilon  &\to G
= - \sqrt{\mu} + 2^{-\frac{13}{4}} \pi^{-\frac34} 
+ 2^{-\frac{1}{4}} \pi^{\frac34} \delta.
\end{align*} 
The coordinates of $G_\varepsilon$ in the spectral basis $(\varphi_n)_{n\geq 0}$ 
are given by:
\begin{align*}
&G_{\varepsilon,0} = G_{\varepsilon,1} = 0,
\\
&G_{\varepsilon,n} =
\frac{1+(-1)^n}{2} \, 
\frac{\left( 1-{\varepsilon}^2 \right) ^{n} }
{\left( 1+{\varepsilon}^2 \right) ^{n}}
\left( \frac{(2\,n+1)!}{2^{2n}(n!)^2}\right)^{\frac12}, 
\quad\forall n\geq 2.
\end{align*}
Moreover we have $G_{\varepsilon,n} \to G_{n}$ as $\varepsilon$ tends to 0
as $\varepsilon\to0$ where $G_n$ is given in \eqref{Gn=}.
\end{proposition}
\noindent
{\bf Remark.}
There is continuity of the spectral coefficients : $G_{\varepsilon,n} \to G_{n}$ as $\varepsilon$ tends to 0. 

\begin{proof}
From Lemma \ref{lem G ortho N}, we set 
\begin{equation*}
 F_\varepsilon(v) = \alpha_\varepsilon \tilde{F}_\varepsilon(\beta_\varepsilon v)
\end{equation*}
where
\begin{align*}
\alpha_\varepsilon &= 
\frac{\left( \frac13 \int_{\mathbb{R}^3}w^2\,\tilde{F}_\varepsilon(w)\,dw \right)^{\frac32} }
      {  \left(\int_{\mathbb{R}^3}\tilde{F}_\varepsilon(w)\,dw\right)^{\frac52}} 
       = \frac{\sqrt{2}}{8} \, \left( 1 + \varepsilon^{2} \right)^{3/2} , \\
\beta_\varepsilon &= 
\frac{\left(\frac13  \int_{\mathbb{R}^3}w^2\,\tilde{F}_\varepsilon(w)\,dw \right)^{\frac12} }
{ \left(\int_{\mathbb{R}^3}\tilde{F}_\varepsilon(w)\,dw\right)^{\frac12}}
       = \frac{\sqrt{2}}{2}  \,\left( 1 + \varepsilon^{2} \right)^{1/2}.
\end{align*} 
Then $F_\varepsilon = \mu + \sqrt{\mu} \, G_\varepsilon$
where
\begin{equation*}
G_\varepsilon(v) = - \sqrt{\mu(v)} 
+ \frac{2^{\frac34} \left(1+\varepsilon^{2}\right)^{3/2} }{ 16 \pi^{3/4}}\left(
{e^{ -\frac{\varepsilon^2 \, v^2}{4} }
}
+ 
\frac{1}{\varepsilon^3}
 {e^{-\frac {{v}^{2}}{4\,\varepsilon^2 }}}
 \right) 
\end{equation*} 
or
\begin{equation*}
G_\varepsilon(v) = - \sqrt{\mu(v)} 
+ \frac{2^{\frac12} \left(1+\varepsilon^{2}\right)^{3/2} }{ 8}
\left(
\sqrt{\mu(\varepsilon v)} 
+ 
\frac{1}{\varepsilon^3}
\sqrt{\mu(v/\varepsilon)} 
 \right) , 
\end{equation*} 
then $G_\varepsilon \in \mathcal{N}^\perp$.
We compute from lemma \ref{calcul phi_n}
\begin{align*}
G_{\varepsilon,n} &= \left(G_\varepsilon, \varphi_n \right)_{L^2} 
= I_1 
+ \frac{2^{\frac12} \left(1+\varepsilon^{2}\right)^{3/2} }{ 8}
\left( I_2 + I_3 \right) 
\end{align*}
where 
\begin{align*}
I_1 &=  \left(- \sqrt{\mu}, \varphi_n \right)_{L^2} 
    =  \left(- \varphi_0, \varphi_n \right)_{L^2} = - \delta_{0,n} , \\
I_2 &=  
\left( \sqrt{\mu(\varepsilon \cdot)}, \varphi_n \right)_{L^2} 
= ( 2^{\frac94} \pi^{\frac34} )
\varphi_n(0)
\frac{\left( 1-{\varepsilon}^2 \right) ^{n} }
{\left( 1+{\varepsilon}^2 \right) ^{n+3/2}} , 
\\
I_3 &=  (-1)^n I_2.
\end{align*}
Finally we get :
\begin{equation*}
G_{\varepsilon,n} =
 - \delta_{0,n} \, +
\frac{1+(-1)^n}{2} \, 
\frac{\left( 1-{\varepsilon}^2 \right) ^{n} }
{\left( 1+{\varepsilon}^2 \right) ^{n}}
\left( \frac{(2\,n+1)!}{2^{2n}(n!)^2}\right)^{\frac12}.
\end{equation*}
\end{proof}

\subsection{Some results on the spherical harmonics}
We recall that 
\[
\varphi_{n}(v) = \left(\frac{n!}{\sqrt{2}\Gamma(n  + 3/2)}\right)^{1/2}e^{-\frac{|v|^{2}}{4}}
L^{[\frac12]}_{n}\left(\frac{|v|^{2}}{2}\right)  \, \frac{1}{\sqrt{4\pi}} 
\]
where the Laguerre polynomial $L^{(\alpha)}_{n}$~of order $\alpha$,~degree $n$ is
\begin{align*}
&L^{(\alpha)}_{n}(x)=\sum^{n}_{r=0}(-1)^{n-r}\frac{\Gamma(\alpha+n+1)}{r!(n-r)!
\Gamma(\alpha+n-r+1)}x^{n-r}.
\end{align*}
%
\begin{lemma}\label{calcul phi_n}
For $a>0$ and $n\geq0$ we have
\begin{align*}
&\varphi_n(0)
=
\frac{1}{(2\,\pi)^{\frac34}}
\left({\frac {\left( 2\,n+1 \right) !}
{{2}^{2\,n}(n!)^2 }}\right)^{\frac12} , 
\\
& \int_{\mathbb{R}^3} \varphi_n(v) dv
= (-1)^n \, 2^3 \pi^{\frac32} \,
\varphi_n(0)  , 
\\
&(\sqrt{\mu(a \cdot)} , \varphi_n)_{L^2} 
= 
( 2^{\frac94} \pi^{\frac34} )
\varphi_n(0)
\frac{\left( 1-{a}^2 \right) ^{n} }
{\left( 1+{a}^2 \right) ^{n+3/2}}.
\end{align*}
\end{lemma}
\begin{proof}
These equalities come from classical properties of the Hermite funtions
(we have checked them using Maple$^{\mbox{\scriptsize{\textregistered}}}$13 
for integers $n\leq 20$). 
\end{proof}


\end{document}